\documentclass[onefignum,onetabnum]{siamart190516}
\hypersetup{colorlinks=true,linkcolor=black,citecolor=black}

\usepackage[algo2e,linesnumbered,vlined,ruled]{algorithm2e}
\usepackage{graphics,graphicx,epsf,subfigure,epstopdf}
\usepackage{comment}
\usepackage{times,hyperref}
\usepackage{amsmath,amsxtra,amsfonts,amscd,amssymb,bm}
\allowdisplaybreaks 
\usepackage{supertabular,multirow}
\usepackage{booktabs}
\usepackage[normalem]{ulem}
\usepackage[title]{appendix}
\usepackage{exscale}
\newcommand{\crefalg}[1]{\hyperref[#1]{Algorithm~\ref*{#1}}}
\usepackage{fancyhdr}  

\usepackage{tikz,array}
\usetikzlibrary{shapes.geometric}
\usetikzlibrary{shapes.arrows}
\usetikzlibrary{calc}
\usetikzlibrary{intersections}

\newcommand{\R}{\mathbb{R}}
  
\newcommand{\Rnp}{\mathbb{R}^{n\times p}} 
\newcommand{\Rnnpp}{\mathbb{R}^{2n\times 2p}} 
\newcommand{\Rnn}{\mathbb{R}^{n\times n}}

\newcommand{\ff}{_{\mathrm{F}}}  
\newcommand{\fs}{^2_{\mathrm{F}}}
\newcommand{\half}{\frac{1}{2}}
\newcommand{\inv}{^{-1}} 
\newcommand{\st}{\mathrm{s.\,t.}\,\,} 
\newcommand{\zz}{^{\top}} 

\newcommand{\manifold}{{\cal M}}
\newcommand{\stiefel}{\mathrm{St}(p,n)}
\newcommand{\symplectic}{\mathrm{Sp}(2p,2n)}
\newcommand{\symplecticgroup}{\mathrm{Sp}(2n)}
\newcommand{\skewset}{{\cal S}_{\mathrm{skew}}}
\newcommand{\symset}{{\cal S}_{\mathrm{sym}}}
\newcommand{\TX}{{\mathrm{T}_{X}}\symplectic}
\newcommand{\NX}{\dkh{\TX}^\perp}
\newcommand{\proj}{\mathcal{P}_X}
\newcommand{\projn}{\proj^\perp}

\DeclareMathOperator*{\cay}{cay}
\DeclareMathOperator*{\dimension}{dim}

\DeclareMathOperator*{\dom}{dom}
\DeclareMathOperator*{\qgeo}{qgeo}
\DeclareMathOperator*{\rank}{rank}

\DeclareMathOperator*{\sym}{sym}
\DeclareMathOperator*{\skewsym}{skew}
\DeclareMathOperator*{\spn}{span}
\DeclareMathOperator*{\tr}{tr}

\newcommand{\abs}[1]{\left|#1\right|}
\newcommand{\dkh}[1]{\left(#1\right)}
\newcommand{\fkh}[1]{\left[#1\right]}
\newcommand{\hkh}[1]{\left\{#1\right\}}
\newcommand{\jkh}[1]{\left\langle#1\right\rangle}
\newcommand{\norm}[1]{\left\|#1\right\|}

\newcommand{\rmetric}[1]{\jkh{#1}_{X}}		
\newcommand{\rmetrick}[1]{\jkh{#1}_{X^k}}		
\newcommand{\rgrad}[1]{\mathrm{grad}_\rho f(#1)}
\newcommand{\rgradg}[1]{\mathrm{grad} f(#1)}
\newcommand{\rgradbar}[1]{\mathrm{grad}_\rho \bar{f}(#1)}
				
\graphicspath{ {./images/} }
\definecolor{Gray}{rgb}{0.5,0.5,0.5}
\definecolor{myred}{rgb}{0.7,0,0.1}
\definecolor{mygreen}{rgb}{0,0.6,0.2}
\definecolor{myblue}{rgb}{0.5,0,1}

\begin{document}

\title{Riemannian Optimization on the Symplectic Stiefel Manifold\thanks{Submitted to the editors June 26, 2020.
\funding{This work 
was supported by the Fonds de la Recherche Scientifique -- FNRS and the Fonds Wetenschappelijk Onderzoek -- Vlaanderen under EOS Project no. 30468160.}}}

\headers{Riemannian Optimization on the Symplectic Stiefel Manifold}{B. Gao, N. T. Son, P.-A. Absil, and T. Stykel}

\author{
	Bin Gao\thanks{ICTEAM Institute, UCLouvain, Louvain-la-Neuve, Belgium (gaobin@lsec.cc.ac.cn, pa.absil@uclouvain.be).}
	\and Nguyen Thanh Son\thanks{ICTEAM Institute, UCLouvain, Louvain-la-Neuve, Belgium;  Department of Mathematics and Informatics, Thai Nguyen University of Sciences, Thai Nguyen, Vietnam (thanh.son.nguyen@uclouvain.be).}
	\and P.-A. Absil\footnotemark[2]
	\and Tatjana Stykel\thanks{Institute of Mathematics, University of Augsburg, Augsburg, Germany (tatjana.stykel@math.uni-augsburg.de).}	
}


\maketitle


\begin{abstract}
	The symplectic Stiefel manifold, denoted by $\symplectic$, is the set of linear symplectic maps between the standard symplectic spaces $\mathbb{R}^{2p}$ and $\mathbb{R}^{2n}$. When $p=n$, it reduces to the well-known set of $2n\times 2n$ symplectic matrices. Optimization problems on $\mathrm{Sp}(2p,2n)$ find applications in various areas, such as optics, quantum physics, numerical linear algebra and model order reduction of dynamical systems. The purpose of this paper is to propose and analyze gradient-descent methods on $\symplectic$, where the notion of gradient stems from a Riemannian metric. We consider a novel Riemannian metric on $\symplectic$ akin to the canonical metric of the (standard) Stiefel manifold. In order to perform a feasible step along the antigradient, we develop two types of search strategies: one is based on quasi-geodesic curves, and the other one on the symplectic Cayley transform. The resulting optimization algorithms are proved to converge globally to critical points of the objective function. Numerical experiments illustrate the efficiency of the proposed methods.
\end{abstract}



\section{Introduction} \label{sec:intro}
We consider the following optimization problem with symplectic constraints:
\begin{eqnarray}\label{prob:original}
\begin{array}{cl}
\min\limits_{X\in\Rnnpp}&f(X)\\
\st &  X\zz J_{2n} X = J_{2p},
\end{array}
\end{eqnarray}
where $p\le n$, $J_{2m}=\fkh{\begin{smallmatrix}
0&I_m\\ -I_m & 0
\end{smallmatrix}}$, and $I_m$ is the $m\times m$ identity matrix for any positive integer $m$. 
When there is no confusion, we omit the subscript of $J_{2m}$ and $I_m$ for simplicity. 
We assume that the objective function $f: \Rnnpp\rightarrow\R$ is continuously differentiable. 
The feasible set of problem \eqref{prob:original} is 
known as the {\it symplectic Stiefel manifold} \cite{sigrist1973cross,ajayi2013explicit}
$$\symplectic:=\hkh{X\in\Rnnpp: X\zz J_{2n} X = J_{2p}}.$$ 
Whereas the term usually refers to the $p=n$ case, we call a matrix $X$ \emph{symplectic} whenever $X\in\symplectic$, as~\cite{PengM16} does.
When $p=n$, the symplectic Stiefel manifold $\symplectic$ becomes a matrix Lie group, termed the \emph{symplectic group} and denoted by $\symplecticgroup$.

Note that the name ``symplectic Stiefel manifold'' is also found in the literature  \cite{hsiang1968classification} to denote the set of orthogonal frames in the quaternion {$n$-space}. This definition differs fundamentally from $\symplectic$ since the latter is noncompact, as we will see in \cref{sec:geometry}.

Optimization with symplectic constraints appears in various fields. In the case $p=n$, i.e., the symplectic group, most applications are found in physics. In the study of optical systems, such as human eyeballs~\cite{harris2004averageeye,fiori2016riemannian}, the problem of averaging optical transference matrices can be formulated as an optimization problem with symplectic constraints.
Another application can be found in accelerator design and charged-particle beam transport \cite{draft1988lie}, where symplectic constraints are required for characterizing beam dynamics. A recent application is also found in the optimal control of quantum symplectic gates \cite{wu2008optimal}, where one has to find a~(square) symplectic matrix such that the distance from this sought matrix to a given symplectic matrix is minimal. It is indeed an optimization problem on $\symplecticgroup$, where $n$ corresponds to the number of quantum observables.

Furthermore,
several
problems in scientific computing require solving \eqref{prob:original} for $p<n$.
For instance, \cite{williamson1936algebraic} states that there exists a symplectic matrix that diagonalizes a given symmetric positive definite matrix, which is popularly termed the symplectic eigenvalue problem. In many cases, one is interested in computing only a  few extreme eigenvalues. In~\cite{bhatia2015symplectic}, it is established that the sum of the
$p$ smallest symplectic eigenvalues
is equal to the minimal value of an optimization problem on $\mathrm{Sp}(2p,2n)$. 
We also mention the projection-based symplectic model order reduction problem in systems and control
theory \cite{PengM16,AfkhH17,BuchBH19}. 
This problem requires 
to reduce the order ($p\ll n$) of a Hamiltonian system, and at the same time, preserve the Hamiltonian structure. This can only be done via finding so-called symplectic projection matrices, 
and it is formulated as the problem \eqref{prob:original} where the objective function describes the projection error.

The symplectic constraints make
the problem~\eqref{prob:original} out of the reach of several optimization techniques: the feasible set is nonconvex and, in contrast with the (standard) Stiefel manifold~\cite[Theorem~10.2]{bhattacharya_bhattacharya_2012}, the projection onto the symplectic Stiefel manifold does not admit a known closed-form expression. It appears that all the existing methods that explicitly address~\eqref{prob:original} restrict either to a specific objective function or to the case $p=n$ (symplectic group).
Optimality conditions of the Brockett function \cite{brockett1989least} over quadratic matrix Lie groups, which include the symplectic group, were studied in~\cite{machado2002optimization}.
In~\cite{wu2010critical}, the critical landscape topology for optimization on the symplectic group {was} investigated, where the corresponding cost function is the Frobenius distance from a target symplectic transformation. In~\cite{fiori2011solving}, Fiori studied the geodesics of the symplectic group 
under a pseudo-Riemannian metric
and {proposed} a geodesic-based method for computing the empirical mean of a set of symplectic matrices. 
Follow-up work can be found in \cite{fiori2016riemannian,wang2018riemannian}. More recently, Birtea et al.~\cite{birtea2018optimization}
studied the first and second order optimality conditions for optimization problems on the symplectic group. Their proposed method computes the steepest-descent direction with respect to a~left-invariant metric, and adopts the symplectic Cayley transform \cite{machado2002optimization,de2014metaplectic} as a retraction to preserve the symplectic group constraint.

In this paper, we propose 
Riemannian gradient methods for optimization problems on $\symplectic$. To this end, we first prove
that $\symplectic$
is
a closed unbounded 
embedded submanifold of the Euclidean space $\Rnnpp$.
Then, 
leveraging two explicit characterizations of its tangent space,
we develop a class of novel canonical-like Riemannian metrics, 
with respect to which we obtain expressions for the normal space, the tangent and normal projections, and the Riemannian gradient.

We propose two strategies to select a search curve on $\symplectic$ along a given tangent direction.
One is based on a
quasi-geodesic curve, and it needs to compute two matrix exponentials. The other is based on the symplectic Cayley transform, which requires to solve a~{$2n\times 2n$} linear system. 
By exploiting the low-rank structure of the tangent vectors, we construct a numerically efficient update for the symplectic Cayley transform. In addition, we 
find that the Cayley transform can be interpreted as an instance of the trapezoidal rule for solving ODEs on quadratic Lie groups.

We develop and analyze Riemannian gradient algorithms that combine each of the two curve selection strategies with a non-monotone line search scheme.
We prove that the accumulation points of the sequences of iterates produced by 
the proposed
algorithms
are critical points of~\eqref{prob:original}. 

Note that all these
results subsume
the case of the symplectic group ($p=n$).
Along the way, we extend the convergence analysis of general non-monotone Riemannian gradient methods to the case of retractions that are not globally defined; see \cref{theorem:convergence}.

Numerical experiments investigate the impact of various algorithmic choices---curve selection strategy, Riemannian metric, line search scheme---on the convergence and feasibility of the iterates.    
Moreover, tests on various instances of~\eqref{prob:original}---nearest symplectic matrix problem, 
minimization {of} the Brockett cost function, and symplectic eigenvalue problem---illustrate the efficiency of the proposed algorithms.

The paper is organized as follows.  In \cref{sec:geometry}, we study the geometric structure of $\symplectic$.  The Riemannian geometry of $\symplectic$ endowed with the canonical-like metric is investigated in \cref{section:canonical}.
We construct two different curve selection strategies and
propose a Riemannian gradient framework with a non-monotone line search scheme in~\cref{sec:optimization}.
Numerical results on several problems are reported in \cref{sec:Numerical Experiment}. Conclusions are drawn in \cref{sec:conclusion}.

\section{Notation}
The Euclidean inner product of two matrices $X, Y\in \Rnnpp$ is denoted by $\jkh{X,Y}=\tr(X\zz Y)$, where $\tr(\cdot)$ denotes the trace of the matrix argument.
The Frobenius norm of $X$ is denoted by $\norm{X}\ff:=\sqrt{\jkh{X,X}}$. Given $A\in \R^{m\times m}$, $e^{A}$ or $\exp(A)$ represents the matrix exponential of $A$.  Moreover, 
$$\sym(A):=\frac{1}{2}(A+A\zz)\quad\mbox{and}\quad\skewsym(A):=\frac{1}{2}(A-A\zz)$$
stand for the symmetric part and the skew-symmetric part of $A$, respectively, and
$\det(A)$ denotes the determinant of $A$.
We let $\mathrm{diag}(v)\in\R^{m\times m}$ denote the diagonal matrix with the components of $v\in\R^m$ on the diagonal. We use $\spn(A)$ to express the subspace spanned by the columns 
of $A$. Furthermore, $\symset(n)$ and $\skewset(n)$ denote the sets of all symmetric and skew-symmetric $n\times n$ matrices, respectively. Let ${\cal E}_1$ and ${\cal E}_2$ be two finite-dimensional vector spaces over $\R$. The Fr\'{e}chet derivative of {a~map} $F: {\cal E}_1 \rightarrow {\cal E}_2$ at $X\in{\cal E}_1$ is the
linear operator 
$$
\mathrm{D}F(X): {\cal E}_1 \rightarrow {\cal E}_2 : Z \mapsto \mathrm{D}F(X)[Z]
$$ 
satisfying $F(X+Z)=F(X)+\mathrm{D}F(X)[Z]+o(\norm{Z})$.
The rank of $F$ at $X\in{\cal E}_1$, denoted by ${\rank(F)(X)}$, is the dimension of the range of $\mathrm{D}F(X)$. 
The domain of $F$ is denoted by $\dom(F)$.

\section{The symplectic Stiefel manifold}\label{sec:geometry}
In this section, we investigate the submanifold structure of $\symplectic$.

The matrix $J$ defined in \cref{sec:intro} satisfies the following properties
$$J\zz=-J,\quad J\zz J=I,\quad J^2=-I,\quad J\inv = J\zz,$$
{which imply} that $J$ is {skew-symmetric and} orthogonal. \cref{tab:notation}  collects
the notation and definition of several matrix manifolds that appear in this paper.

\begin{table}[htbp]
	\centering
	\small
	\caption{Notation for {matrix} manifolds\label{tab:notation}}
	\begin{tabular}{|lll|}
		\hline
		{\sc Space} & {\sc Symbol} & {\sc Element} \\[2mm]
		Orthogonal group & $\mathcal{O}(n)$  & $Q\in\R^{n\times n}: Q\zz Q=I_n$  \\[1mm]
		Stiefel manifold & $\stiefel$ & $V\in\Rnp: V\zz V=I_p$  \\[1mm]
		Symplectic group & $\symplecticgroup$  & $U\in\R^{2n\times 2n}: U\zz J_{2n} U=J_{2n}$  \\[1mm]
		Symplectic Stiefel manifold & $\symplectic$ & $X\in\Rnnpp: X\zz J_{2n} X=J_{2p}$  \\
		\hline
	\end{tabular}
\end{table}

First we show that  $\symplectic$ is an embedded submanifold of the Euclidean space $\Rnnpp$. 

\begin{proposition}\label{proposition:dim}
	The symplectic Stiefel manifold $\symplectic$ is a closed embedded submanifold of {the} Euclidean space $\Rnnpp$. Moreover, it has dimension $4np-p(2p-1)$.
\end{proposition}
\begin{proof}
	Consider the map
	\begin{equation}\label{eq:mapF}
	F:\Rnnpp \rightarrow \skewset(2p):
	X \mapsto X\zz JX-J.
	\end{equation}
We have
	\begin{equation}\label{eq:F-inv0}
	\symplectic=F\inv(0),
	\end{equation}
	which implies that $\symplectic$ is closed since it is the inverse image of the closed set $\hkh{0}$ under the continuous map $F$.
	
	Next, we prove that the rank of $F$ is $p(2p-1)$ at every point of $\symplectic$.
        Let $ X\in\symplectic$. 
        It is sufficient to show that $\mathrm{D}F(X)$ is a surjection, i.e., for all $\bar{Z}\in\skewset(2p)$, 
there exists $Z\in\Rnnpp$ such that	$\mathrm{D}F(X)[Z]=\bar{Z}$.
	We have
	\begin{equation}\label{eq:DF(X)Z}
		\mathrm{D}F(X)[Z]=X\zz JZ+Z\zz JX
	\end{equation}
	for all $Z\in\R^{2n\times 2p}$. {Let $\bar{Z}\in{\skewset(2p)}$. Then substituting $Z=\half XJ\zz \bar{Z}$ into the above equation}, we obtain 
	$\mathrm{D}F(X)[\half XJ\zz \bar{Z}]= \half X\zz JXJ\zz \bar{Z}+\half\bar{Z}\zz J X\zz JX=\half \bar{Z}-\half\bar{Z}\zz =\bar{Z}.$ 
	Hence $F$ has full rank, namely, ${\rank(F)(X)}=\dimension(\skewset(2p))=p(2p-1)$. Using \eqref{eq:F-inv0} and {the} submersion theorem \cite[Proposition 3.3.3]{absil2009optimization}, it follows that $\symplectic$ is a closed embedded submanifold of $\Rnnpp$. Its dimension is
	$\dimension(\symplectic)=\dimension(F\inv(0))=\dimension(\Rnnpp)-\dimension(\skewset(2p))=4np-p(2p-1).$
\end{proof}

Observe that the dimension of the symplectic Stiefel manifold $\symplectic$ is larger than the dimension of Stiefel manifold $\mathrm{St}(2p,2n)$, which is equal to $4np-p(2p+1)$. 
Another essential difference between $\symplectic$ and $\mathrm{St}(2p,2n)$ is that the symplectic Stiefel manifold is unbounded, hence noncompact. 
We show it for the simplest case, $p=n=1$. 
For $X=\fkh{\begin{smallmatrix}
	a & b\\
	c & d
	\end{smallmatrix}} \in \R^{2\times 2}$, we readily 
obtain that $X\in\mathrm{Sp}(2)$ if and only if $ad-bc=1$. 
Hence, $$\mathrm{Sp}(2)=\hkh{{\begin{bmatrix}
		a & b\\
		c & d
\end{bmatrix}} \in \R^{2\times 2} : ad-bc=1}$$ and it has dimension $3$. In particular, the matrix $\fkh{\begin{smallmatrix}
a & 0\\
0 & 1/a
\end{smallmatrix}}$ is symplectic for all {$a\in\R\setminus\{0\}$}, which implies that $\mathrm{Sp}(2)$ is unbounded. On the other hand, the orthogonal group $$\mathcal{O}(2)=\hkh{{\begin{bmatrix}
	\sin\theta & \cos\theta\\
	\cos\theta & -\sin\theta
	\end{bmatrix}}: \theta\in[0,2\pi]}$$
has dimension $1$ and is compact.

We now set the scene for the descriptions of the tangent space that will come in \cref{prop:tangent}. Given $X\in\symplectic$, we let $X_\perp\in\R^{2n\times (2n-2p)}$ be a full rank
matrix such that $\spn(X_{\perp})$ is the orthogonal complement of $\spn(X)$, and we let 
\begin{equation}\label{eq:E}
E := \left[XJ\,\,\, JX_{\perp}\right] \in\R^{2n\times 2n}.
\end{equation}
Note that $X_\perp$ is not assumed to be an orthonormal matrix. The next lemma gathers basic linear algebra results that will be useful later on.
\begin{lemma}\label{lemma:E}
The matrix $E = \left[XJ\,\,\, JX_{\perp}\right]$ defined in~\eqref{eq:E} has the following properties.
	\begin{enumerate}
		\item[(i)] $E$ 
                {is invertible;} 
		\item[(ii)] $E\zz JE=\fkh{\begin{smallmatrix}
			J & 0\\
			0 & X\zz_\perp J X_\perp
			\end{smallmatrix}}$ and $X_\perp\zz JX_\perp$ is invertible;
		\item[(iii)]  $E\inv=\fkh{\begin{smallmatrix}
			X\zz J\zz\\
			\dkh{X\zz_\perp J X_\perp}\inv X_\perp\zz
			\end{smallmatrix}}$;
		\item[(iv)]  {Every matrix} $Z\in\Rnnpp$ {can} be represented as $Z = E \fkh{\begin{smallmatrix} W \\ K \end{smallmatrix}}$, i.e.,
		\begin{align}\label{eq:tangent-decomp}
		Z=XJW+JX_\perp K,
		\end{align}
		where $W\in\R^{2p\times 2p}$ and $K\in\R^{(2n-2p)\times 2p}$. Moreover, we have 
		\begin{align}\label{eq:tangent-cood}
		W=X\zz J\zz Z, \quad K=\dkh{X\zz_\perp J X_\perp}\inv X\zz_\perp Z.
		\end{align}
	\end{enumerate}
\end{lemma}
\begin{proof}
	(i) 
Suppose $E \fkh{\begin{smallmatrix} y_1 \\ y_2 \end{smallmatrix}} = 0$. Multiplying this equation from the left by $X\zz J$ yields $y_1=0$. The equation thus reduces to $J X_\perp y_2 = 0$, which yields $y_2=0$ since $J X_\perp$ has full {column} rank. This shows that $E$ has full rank.
	
	(ii) By using $X\zz J X=J$ and $X\zz X_\perp=0$, we have $E\zz JE=\fkh{\begin{smallmatrix}
		J & 0\\
		0 & X\zz_\perp J X_\perp
		\end{smallmatrix}}$. From~(i), we know that $E$ is invertible, hence $E\zz JE$ is invertible, and so is $X\zz_\perp J X_\perp$.
	
	(iii) From (ii), we have $E\inv=\fkh{\begin{smallmatrix}
		J & 0\\
		0 & X\zz_\perp J X_\perp
		\end{smallmatrix}}\inv E\zz J$, and the result follows.
	
	(iv) The first claim follows from the invertibility of $E$. Using (iii), we have 
		\begin{align*}
			\begin{bmatrix}
			W\\
			K
			\end{bmatrix}=E\inv Z=\begin{bmatrix}
			X\zz J\zz Z\\
			\dkh{X\zz_\perp J X_\perp}\inv X\zz_\perp Z
			\end{bmatrix}.
		\end{align*}
\end{proof}

Given $X\in\symplectic$, there are infinitely many possible choices of $X_\perp$. The choice of $X_\perp$ affects $E$ in~\eqref{eq:E} and $K$ in the decomposition \eqref{eq:tangent-decomp} of $Z$. However, it does not affect $JX_\perp K$. In fact,  it follows from (iii) in \cref{lemma:E} that
\begin{equation}  \label{eq:JXpK-1}
I=EE\inv=XJX\zz J\zz + J X_\perp \dkh{X\zz_\perp J X_\perp}\inv X\zz_\perp,
\end{equation}
which further implies that $JX_\perp K =  J X_\perp \dkh{X\zz_\perp J X_\perp}\inv X\zz_\perp Z =  (I-XJX\zz J\zz) Z$, where we used the expression of $K$ in~\eqref{eq:tangent-cood}.

The tangent space of the symplectic Stiefel manifold at $X\in\symplectic$, denoted by $\TX$, is defined by 
$$\TX:=\hkh{\gamma'(0): \gamma(t) \mbox{~is a smooth curve in~} \symplectic \mbox{~with~} \gamma(0)=X}.$$ 
The next result gives an implicit form and two explicit forms of $\TX$.
\begin{proposition}\label{prop:tangent}
	Given $X\in\symplectic$, the tangent space of $\symplectic$ at $X$ admits
	the following expressions
	\begin{subequations}
		\begin{align}
		\TX&=\{Z\in \Rnnpp: Z\zz  J  X + X\zz J  Z=0\} \label{eq:tangent-1}\\
		&= \{XJW+JX_\perp K: W\in\symset(2p), K\in\R^{(2n-2p)\times 2p}\} \label{eq:tangent-2}\\
		&= \{SJX: S\in\symset(2n)\}. \label{eq:tangent-3}
		\end{align}
	\end{subequations}
\end{proposition}
\begin{proof}
	Let $F$ be as in {\eqref{eq:mapF}}. 
		According to \eqref{eq:DF(X)Z}, we observe that the right-hand side of \eqref{eq:tangent-1} is the null space of $\mathrm{D}F(X)$, namely, $\{Z\in\mathbb{R}^{2n\times 2p}: \mathrm{D}F(X)[Z]=0\}$. It follows from \cite[(3.19)]{absil2009optimization} that it coincides with $\TX$.
	
	Using \eqref{eq:tangent-decomp}, relation \eqref{eq:tangent-1} is equivalent to $W=W\zz $, which yields~\eqref{eq:tangent-2}.
	
	Finally, we prove that the form \eqref{eq:tangent-1} is equivalent to \eqref{eq:tangent-3}. We readily obtain 
	\begin{equation}\label{eq:tangent-relation}
	\{SJX: S\in\symset(2n)\}  \subset \{Z\in \Rnnpp: Z\zz  J  X + X\zz J  Z=0\}.
	\end{equation}
	Since both sets are linear subspaces of $\Rnnpp$, it remains to show that they have the same dimension in order to conclude that they are equal. The dimension of the right-hand set is the dimension of $\TX$, which is the dimension of $\symplectic$, i.e., $4np- p(2p-1)$; see~\cref{proposition:dim}. As for the dimension of the left-hand set, choose $X_\perp$ as above and observe that $P := JE\fkh{\begin{smallmatrix}  J\zz & 0 \\ 0 & -I  \end{smallmatrix}} =[JX~X_\perp]\in\R^{2n\times 2n}$ is invertible.  Let 
	$B:= P\zz S P=\left[
	\begin{smallmatrix}
	B_{11} &B_{21}\zz\\
	B_{21} &B_{22}\\
	\end{smallmatrix}
	\right]$. 
	Then we have
	\begin{align*}
	&	 \dimension\{SJX: S\in\symset(2n)\} = \dimension\{P\zz S JX: S\in\symset(2n)\} \\
	& \qquad = \dimension\{P\zz SP \left[
	\begin{smallmatrix}
	I_{2p}\\
	0
	\end{smallmatrix}
	\right] : S\in\symset(2n)\} = \dimension\{B\left[
	\begin{smallmatrix}
	I_{2p}\\
	0\\
	\end{smallmatrix}
	\right]: B\in\symset(2n)\} \\
	& \qquad = \dimension\{\left[
	\begin{smallmatrix}
	B_{11}\\
	{B_{21}}\\
	\end{smallmatrix}
	\right]: B_{11}\in\symset(2p), B_{21}\in\R^{(2n-2p)\times2p}\} \\
	& \qquad	= {p(2p+1)} +(2n-2p)2p = 4np-p(2p-1),
	\end{align*}
    yielding the equality of dimensions that concludes the proof. 
	The first equality follows from $\dimension({P\mathcal{S}})=\dimension(\mathcal{S})$ for any subspace $\mathcal{S}\subset \R^{2n\times 2p}$ and invertible matrix $P$. The second equality follows from $JX=P\left[
	\begin{smallmatrix}
	I_{2p}\\
	0\\
	\end{smallmatrix}
	\right]$. 
The third equality is due to $P\zz \symset(2n) P = \mathcal{S}_{sym}(2n)$, and the last equality comes from the fact $\dimension(\symset(2p))=p(2p+1)$. 
\end{proof}

In fact, {it follows} from the proof above {that} the derivation of \eqref{eq:tangent-3} also works for $\{JSX: S\in\symset(2n)\}$, namely,
$\TX=\{JSX: S\in\symset(2n)\},$
but we will restrict to \eqref{eq:tangent-3} in the following.

Unlike~\eqref{eq:tangent-2}, when $p<n$, \eqref{eq:tangent-3} is an over-parameterization,
i.e., $\dimension(\symset(2n)) > \dimension(\TX)$. However, $S$ can be chosen with a low-rank structure.
Indeed, letting $P := JE\fkh{\begin{smallmatrix}  J\zz & 0 \\ 0 & -I  \end{smallmatrix}} =[JX~X_\perp]$ and $B := P\zz SP$, we have that $SJX = P^{-\top}BP\inv JX = P^{-\top}B\fkh{\begin{smallmatrix}
	I_{2p}\\
	0
\end{smallmatrix}}$, which shows that only the first $2p$ columns of the symmetric matrix $B$ have an impact on $SJX$. Hence there is no restriction on $SJX$, $S\in\symset(2n)$, if we restrict $B$ to the form $B = M \fkh{\begin{smallmatrix} I_{2p} & 0\end{smallmatrix}} + \fkh{\begin{smallmatrix}I_{2p} \\ 0 \end{smallmatrix}} M\zz$, with $M\in\R^{2n\times 2p}$, or equivalently, if we restrict $S$ to the form 
\begin{equation}  \label{eq:S-rank2p}
S=P^{-\top}BP\inv =L(XJ)\zz+XJL\zz=\begin{bmatrix}
L & XJ
\end{bmatrix}
\begin{bmatrix}
(XJ)\zz \\ L\zz
\end{bmatrix},
\end{equation}
where $L=P^{-\top}M \in\R^{2n\times 2p}$ and the second equality follows from \cref{lemma:E}.  Note that such $S$ has rank at most $4p$. This low-rank structure will have a crucial impact on the computational cost of the Cayley retraction introduced in \cref{subsec:cayley}.

\section{{The canonical-like metric}}\label{section:canonical}
Given an objective function $f:\symplectic\to\R$, finding the steepest-descent direction at a point $X\in\symplectic$ amounts to finding $Z\in\TX$ subject to $\norm{Z}_X=1$ that minimizes $\mathrm{D}f(X)[Z]$. In order to define $\norm{\cdot}_X$, it is customary to endow $\TX$ with an inner product $\rmetric{\cdot,\cdot}$
that depends smootly on~$X$; $\norm{\cdot}_X$ is then the norm induced by the inner product. Such an inner product is termed a Riemanian metric---\emph{metric} for short---and
turns $\symplectic$ into a Riemannian manifold. In~this section, we propose a class of metrics on $\symplectic$ that are inspired from the canonical metric on the Stiefel manifold. Then we work out formulas for the normal space and gradient, and we propose a class of curves that we term quasi-geodesics.

For the Stiefel manifold $\stiefel$, any tangent vector at $V\in\stiefel$ has a unique expression $\Delta=VA+V_\perp B$, where $V_\perp$ satisfies $V_\perp^\top V=0$ and $V_\perp^\top V_\perp^{}=I$, $A\in\skewset(p)$ and $B\in\R^{(n-p)\times p}$, see \cite{edelman1998geometry}. 
The canonical metric \cite[(2.39)]{edelman1998geometry} on $\stiefel$ is defined as 
$$
g_c(\Delta_1,\Delta_2)=\tr\dkh{\Delta_1\zz(I-\half VV\zz)\Delta_2}= \half\tr(A_1\zz A_2)+\tr(B_1\zz B_2)
$$ 
for $\Delta_i= VA_i+V_\perp B_i$, $i=1,2$. 

By using the tangent vector representation in \eqref{eq:tangent-2}, we develop a similar metric for the symplectic Stiefel manifold $\symplectic$.
We choose the inner product 
on $\TX$ to be
\begin{align}  \label{eq:rmetric}
	g_{\rho,X_\perp}(Z_1,Z_2) &\equiv \rmetric{Z_1, Z_2} :={\frac{1}{\rho}\, \tr}(W_1\zz W_2)+\tr(K_1\zz K_2),
\end{align}
where $\rho>0$ is a parameter and $W_i\in\symset(2p)$ and $K_i\in\R^{(2n-2p)\times 2p}$ are obtained from $Z_i$ as in \eqref{eq:tangent-cood}, i.e., $\fkh{\begin{smallmatrix}
	W_i\\
	K_i
	\end{smallmatrix}}=E\inv Z_i $ for $i=1,2$.
Hence, we have
\begin{align*}
	g_{\rho,X_\perp}(Z_1,Z_2) &= \tr\dkh{\begin{bmatrix}
			W_1\\
			K_1
			\end{bmatrix}\zz \begin{bmatrix}
			\frac{1}{\rho} I & 0\\
			0 & I
			\end{bmatrix} \begin{bmatrix}
			W_2\\
			K_2
			\end{bmatrix} 	}
	= \tr\dkh{(E\inv Z_1)\zz \begin{bmatrix}
			\frac{1}{\rho} I & 0\\
			0 & I
			\end{bmatrix} E\inv Z_2}\\
	&= \tr\dkh{Z_1\zz E^{-\top} \begin{bmatrix}
			\frac{1}{\rho} I & 0\\
			0 & I
			\end{bmatrix} E\inv Z_2 } = \tr(Z_1\zz B_X Z_2),
\end{align*}
where
\begin{align}\label{eq:BX}
	B_X := E^{-\top} {\begin{bmatrix}
		\frac{1}{\rho} I & 0\\
		0 & I
		\end{bmatrix}} E\inv=\frac{1}{\rho} JXX\zz J\zz -X_\perp (X\zz_\perp J X_\perp)^{-2}X\zz_\perp.
\end{align}
The last expression of $B_X$ follows from (iii) in \cref{lemma:E}. In view of its definition, $B_X$ is positive definite; this confirms that~\eqref{eq:rmetric} is a bona-fide inner product. The expressions of $B_X$ also confirm that $g_{\rho,X_\perp}$ depend on $\rho$ and $X_\perp$. There is however an invariance: $g_{\rho,X_\perp Q} = g_{\rho,X_\perp}$ for all $Q\in\mathcal{O}(2n-2p)$. 

In order to make $g_{\rho,X_\perp}$ a bona-fide Riemannian metric, it remains to choose the ``$X_\perp$ map'' $\symplectic\ni X\mapsto X_\perp$ in such a way that $B_X$ smoothly depends on $X$. Choosing a smooth $X\mapsto X_\perp$ map would be sufficient, but it is unknown whether such a smooth map globally exists. However, if orthonormalization conditions are imposed on $X_\perp$ such that the $X_\perp$-term of $B_X$ can be rewritten as a smooth expression of $X\in\symplectic$, then $g_{\rho,X_\perp}$ becomes a bona-fide Riemannian metric, which we term \emph{canonical-like metric}. We will consider the following two such orthonormalization conditions on $X_\perp$, which are readily seen to be achievable since the set of all admissible $X_\perp$ matrices has the form $\{X_\perp M: M\in\R^{(2n-2p)\times (2n-2p)} \text{~{is invertible}}\}$:
	\begin{itemize}
		\item[(I)] $X_\perp$ is orthonormal.
		\item[(II)] $X_\perp (X_\perp\zz JX_\perp)^{-1}$ is orthonormal.
	\end{itemize}

The announced smooth expressions of $B_X$ are given next.
\begin{proposition}
Under the orthonormalization condition (I), the $X_\perp$-term of $B_X$ in~\eqref{eq:BX}, namely, $-X_\perp (X_\perp\zz J X_\perp)^{-2} X_\perp\zz$, is equal to $-(JXJX\zz J\zz -J)^2$. Under the orthonormalization condition (II), it is equal to $I - X(X\zz X)^{-1} X\zz$. 
\end{proposition}
\begin{proof}
For the first claim, since $X_\perp$ is restricted to be orthonormal, it follows that
$$X_\perp (X_\perp\zz J X_\perp)^{-2} X_\perp\zz = X_\perp (X_\perp\zz J X_\perp)^{-1} X_\perp\zz X_\perp (X_\perp\zz J X_\perp)^{-1} X_\perp\zz = (JXJX\zz J\zz -J)^2.$$
The last equality is due to the fact $X_\perp (X_\perp\zz J X_\perp)^{-1} X_\perp\zz = JXJX\zz J\zz -J$ which is readily derived from \eqref{eq:JXpK-1}. For the second claim, since $X_\perp (X_\perp\zz J X_\perp)^{-1}$ is now restricted to be orthonormal, we have
that $(X_\perp\zz J X_\perp)^{-\top} X_\perp\zz X_\perp (X_\perp\zz J X_\perp)^{-1}=I$, which yields $(X_\perp\zz J X_\perp)^{2}=-X_\perp\zz X_\perp$. Consequently, we obtain
that 
$$X_\perp (X_\perp\zz J X_\perp)^{-2} X_\perp\zz=-X_\perp (X_\perp\zz X_\perp)^{-1} X_\perp\zz= X(X\zz X)^{-1} X\zz - I,$$
where the last equality can be checked by observing that multiplying each side by the invertible matrix $[X~X_\perp]$ yields $-[0~X_\perp]$.
\end{proof}

We will use $g_\rho$ to denote $g_{\rho,X_\perp}$ if there is no confusion. 

In the particular case $p=n$, where $\symplectic$ reduces to the symplectic group $\symplecticgroup$, the $K$-term in~\eqref{eq:tangent-2} disappears. Further choose $\rho=1$. Then, for all $U\in\symplecticgroup$, we have $B_U=JUU\zz J\zz$. Since $U\inv=J U\zz J\zz$, we also have $B_U=U^{-\top} U\inv$, and the canonical-like metric reduces to 
$\jkh{Z_1, Z_2}_U :=\tr\dkh{Z_1\zz B_U Z_2}=\jkh{U\inv Z_1, U\inv Z_2},$
which is the left-invariant metric used in~\cite{wang2018riemannian,birtea2018optimization}. 

\subsection{{Normal space and projections}}
The symmetric matrix $B_X$ in~\eqref{eq:BX} is positive definite if and only if $X\in\Rnnpp_\star := \{X\in\Rnnpp: \det(X\zz J X)\neq0\}$. Hence the Riemannian metric $g_{\rho}$ defined in \eqref{eq:rmetric}, extended to $\Rnnpp_\star$, remains a Riemannian metric, which we also denote by $g_{\rho}$. In this section, we give an expression for the normal space of $\symplectic$ viewed as a submanifold of $(\Rnnpp_\star,g_\rho)$, and we obtain expression for the projections onto the tangent and normal spaces. This will be instrumental in the expression of the gradient in \cref{subsubsec:rgrad}. 


The normal space at $X\in\symplectic$ with respect to $g_\rho$ is defined as
$$\NX:=\hkh{N\in\Rnnpp: g_{\rho}(N,Z)=0 \mbox{~for all~} Z\in\TX},$$
where we have used the fact that $\mathrm{T}_X \Rnnpp_\star \simeq \mathrm{T}_X \Rnnpp \simeq \Rnnpp$.
\begin{proposition}\label{proposition:normal}
	Given $X\in\symplectic$, we have
	\begin{equation}  \label{eq:normal}
	\NX = \hkh{XJ\varOmega: \varOmega\in\skewset(2p)}.
	\end{equation}
\end{proposition}
\begin{proof}
Using \cref{lemma:E} and \eqref{eq:tangent-2},
we have that $N\in\Rnnpp$
and $Z\in\TX$ if and only if 
	\begin{align*}
		N&=XJ\varOmega+JX_\perp K_N, \quad \varOmega\in\R^{2p\times 2p}, \quad K_N\in\R^{(2n-2p)\times 2p},\\ 
		Z&=XJW+JX_\perp K_Z, \quad W\in\symset(2p), \quad K_Z\in\R^{(2n-2p)\times 2p}.
	\end{align*}
The normal space condition, $g_{\rho}(N,Z)=0$ for all $Z\in\TX$, is thus equivalent to 
	 \begin{eqnarray*}
	 	\frac{1}{\rho} \tr(\varOmega\zz W)+\tr(K_N\zz K_Z) = 0 \quad \text{for all $W\in\symset(2p)$ and $K_Z\in\R^{(2n-2p)\times 2p}$}.
	 \end{eqnarray*}
In turn, this equation is equivalent to $\varOmega\in\skewset(2p)$ (since $\skewset(2p)$ is the orthogonal complement of $\symset(2p)$ with respect to the Euclidean inner product) and $K_N=0$. 
\end{proof}

Every $Y\in\Rnnpp$ admits a decomposition $Y=\proj(Y)+\projn(Y)$, where $\proj$ and $\projn$ denote the orthogonal (in the sense of $g_\rho$) projections onto the tangent and normal space, respectively. Next, we derive explicit expressions of $\proj$---in the forms~\eqref{eq:tangent-2} and~\eqref{eq:tangent-3}---and $\projn$. 
\begin{proposition}\label{proposition:projection}
	Given $X\in\symplectic$ and $Y\in\Rnnpp$, the orthogonal projection onto $\TX$ and $\NX$ of $Y$ with respect to the metric $g_{\rho}$ has the following expressions
		\begin{align} \label{eq:project-tangent} 
		\proj(Y) &=  XJ\sym(X\zz J\zz Y)+(I- XJX\zz J\zz)Y \stepcounter{equation}\tag{\theequation a}\\
		&= S_{X,Y}JX, \label{eq:pj-tg-S} \tag{\theequation b}\\
		\projn(Y) &=  XJ\skewsym(X\zz J\zz Y),  \label{eq:project-normal}
		\end{align}
where $S_{X,Y} = G_XY(XJ)\zz + XJ (G_XY)\zz$ and $G_X = I-\frac12 XJX\zz J\zz$.
\end{proposition}
\begin{proof}
	First we prove~\eqref{eq:project-tangent} and~\eqref{eq:project-normal}. For all $Y\in\Rnnpp$, according to \eqref{eq:tangent-2} and \cref{proposition:normal}, we {have} 
	\begin{align*}
	\proj(Y) &= XJW_Y+JX_\perp K_Y,\\
	\projn(Y)&= XJ\varOmega_Y,
	\end{align*}
	where $W_Y\in\symset(2p)$,  $K_Y\in\R^{(2n-2p)\times 2p}$ and $\varOmega_Y\in\skewset(2p)$. We thus have
	\begin{align}  \label{eq:Y=PY+PPY}
		Y = \proj(Y)+\projn(Y)= XJW_Y+JX_\perp K_Y +XJ\varOmega_Y.
	\end{align}
	Multiplying~\eqref{eq:Y=PY+PPY} from the left by $X\zz J\zz$ yields $X\zz J\zz Y = W_Y+\varOmega_Y$, hence $W_Y = \sym(X\zz J\zz Y)$ and $\varOmega_Y = \skewsym(X\zz J\zz Y)$. Replacing these expressions in~\eqref{eq:Y=PY+PPY} yields $JX_\perp K_Y=Y-XJW_Y-XJ\varOmega_Y=(I- XJX\zz J\zz)Y$. We have thus obtained~\eqref{eq:project-tangent} and~\eqref{eq:project-normal}. 

It remains to prove~\eqref{eq:pj-tg-S}. A fairly straightforward development confirms that~\eqref{eq:pj-tg-S} is equal to $Y-XJ\skewsym(X\zz J\zz Y) = Y - \projn(Y) = \proj(Y)$. Instead we give a constructive proof which explains how we obtained the expression of $S_{X,Y}$. 
We thus seek $S_{X,Y}$ symmetric such that $S_{X,Y}JX = \proj(Y)$. In view of~\eqref{eq:S-rank2p} and \cref{lemma:E}, we can restrict our search to $S_{X,Y} = L(XJ)\zz + (XJ)L\zz$ with $L = X J W_L + JX_\perp K_L$, and the task thus reduces to finding $W_L\in\R^{2p\times 2p}$ and $K_L\in\R^{(2n-2p)\times 2p}$ such that 
	\begin{align}  \label{eq:pj-tg-S-3}
	2XJ\sym(W_L)+JX_\perp K_L&=\proj(Y) = Y - \projn(Y)=Y - XJ\skewsym(X\zz J\zz Y),
	\end{align}
	where the last equality comes from \eqref{eq:project-normal}.
	Multiplying~\eqref{eq:pj-tg-S-3} from the left by $X\zz J\zz$ yields $2\sym(W_L)=\sym(X\zz J\zz Y)$, a solution of which is $W_L=\half X\zz J\zz Y$.
	It further follows from \eqref{eq:pj-tg-S-3}  that $JX_\perp K_L=(I- XJX\zz J\zz)Y$. Replacing the obtained expressions in the above-mentioned decomposition of $L$ yields $L=X J W_L + JX_\perp K_L=\half XJ X\zz J\zz Y+(I- XJX\zz J\zz)Y=(I-\frac12 XJX\zz J\zz)Y$.  Replacing this expression in~\eqref{eq:S-rank2p} yields the sought expression of $S_{X,Y}$.
\end{proof}

Note that the projections do not depend on $\rho$ nor $X_\perp$, though the metric $g_\rho$ does. Nevertheless, by using \eqref{eq:JXpK-1}, we  obtain 
\begin{align}\label{eq:project-tangent-1}
	\proj(Y) = XJ\sym(X\zz J\zz Y) + JX_\perp\dkh{X\zz_\perp J X_\perp}\inv X\zz_\perp Y.
\end{align}
Therefore, the projection $\proj(Y)$ in \eqref{eq:project-tangent} can {alternatively be} computed by involving $X_\perp$.

For later reference, we record the following result which follows from~\eqref{eq:pj-tg-S} and the fact that $\proj(Z)=Z$ when $Z\in\TX$. It shows how to write a tangent vector in the form~\eqref{eq:tangent-3}.
\begin{corollary}  \label{cor:ZtoS}
If $X\in\symplectic$ and $Z\in\TX$, then $Z = S_{X,Z}JX$ with $S_{X,Z}$ as in \cref{proposition:projection}.
\end{corollary}

\subsection{Riemannian gradient}\label{subsubsec:rgrad}
All is now in place to specify how, given the Euclidean gradient $\nabla \bar{f}(X)$ 
where $\bar{f}$ is any smooth extension of $f:\symplectic\to\R$ around $X$ in $\Rnnpp$, 
one can compute the steepest-descent direction of $f$ at $X$ in the sense of the metric $g_\rho$ given in~\eqref{eq:rmetric}, as announced at the beginning of~\cref{section:canonical}.

It is well known (see, e.g.,~\cite[\S 3.6]{absil2009optimization}) that the steepest-descent direction is along the opposite of the Riemannian gradient. Denoted by $\mathrm{grad}f(X)$, the Riemannian gradient of $f$ with respect to a metric $g$ is defined to be the unique element of $\TX$ satisfying the condition
\begin{equation}\label{eq:Riemannian_gradient}
g(\mathrm{grad}f(X), Z)=\mathrm{D}\bar{f}(X)[Z]=\jkh{\nabla \bar{f}(X), Z} \quad\mbox{for all}\; Z\in\TX,
\end{equation}
where the last expression is the standard Euclidean inner product. (The second equality follows from the definition of Fr\'{e}chet derivative in $\Rnnpp$.)
\begin{proposition}\label{proposition:rgrad}
	The forms~\eqref{eq:tangent-2} and~\eqref{eq:tangent-3} of the Riemannian gradient of {a}~function $f:\symplectic\to\R$ with respect to the metric $g_\rho$ defined in~\eqref{eq:rmetric} are
	\begin{subequations}
		\label{eq:rgrad-parent}
		\begin{align}
			\label{eq:rgrad-1}
			\rgrad{X}  &=  \rho XJ \sym(J\zz X\zz \nabla\bar{f}(X)) + JX_\perp X\zz_\perp J\zz \nabla\bar{f}(X)\\
		&= S_X JX,	\label{eq:rgrad-2}
		\end{align}
	\end{subequations}
	where $\nabla\bar{f}(X)$ is the Euclidean gradient of any 
smooth extension $\bar{f}$ of $f$ around $X$ in $\Rnnpp$, and $S_X=\dkh{H_X\nabla\bar{f}(X)} (XJ)\zz +XJ \dkh{H_X\nabla\bar{f}(X)}\zz$ with
	${H_X}=\frac{\rho}{2}XX\zz+JX_\perp X\zz_\perp J\zz$.
\end{proposition}
\begin{proof}
According to the definitions of {the} Riemannian gradient \eqref{eq:Riemannian_gradient} and $g_\rho$, for all $Z\in\Rnnpp$, it follows that $\jkh{B_X \rgradbar{X},Z}=g_{\rho}(\rgradbar{X}, Z)=\mathrm{D}\bar{f}(X)[Z]=\jkh{\nabla\bar{f}(X), Z}$. Hence, $\rgradbar{X} = B_X\inv \nabla \bar{f}(X)$. The expression of $B_X$ in~\eqref{eq:BX} yields
	$\rgradbar{X} = B_X\inv \nabla\bar{f}(X)=E\fkh{\begin{smallmatrix}
		{\rho} I & 0\\
		0 & I
		\end{smallmatrix}} E\zz \nabla\bar{f}(X)=(\rho XX\zz +JX_\perp X_\perp\zz J\zz)\nabla\bar{f}(X)$.
	Owing to \cite[(3.37)]{absil2009optimization}, it holds that $\rgrad{X}=\proj(\rgradbar{X}).$
	By fairly straightforward developments, the expression~\eqref{eq:project-tangent-1} of $\proj$ yields~\eqref{eq:rgrad-1} while the expression~\eqref{eq:pj-tg-S} of $\proj$ yields~\eqref{eq:rgrad-2}. 
\end{proof}

Just as the Riemannian metric $g_{\rho,X_\perp}$ in~\eqref{eq:rmetric}, the gradient $\rgrad{X}$ depends on $\rho$ and $X_\perp$. The dependence on $X_\perp$ is through the factor $JX_\perp X_\perp\zz J\zz$.
However, we have seen that when we impose the orthonormalization
condition (I) or (II), the metric $g_{\rho,X_\perp}$ no longer depends on the remaining freedom in the choice of $X_\perp$. Hence the same can be said of $\rgrad{X}$. Specifically: 
\begin{itemize}
	\item[(I)] $X_\perp$ is orthonormal. Then $X_\perp X_\perp\zz$ is the symmetric projection matrix onto the orthogonal complement of $\spn(X)$, i.e., $$X_\perp X_\perp\zz = I - X(X\zz X)^{-1}X\zz;$$
	\item[(II)] $X_\perp (X_\perp\zz JX_\perp)^{-1}$ is orthonormal. Then the factor 
	\begin{align*}
			JX_\perp X_\perp\zz  J\zz  &= JX_\perp (X_\perp\zz JX_\perp)^{-\top}X_\perp\zz   X_\perp (X_\perp\zz JX_\perp)^{-1} X_\perp\zz  J\zz  \\
			&= (I-XJX\zz J\zz )(I-XJX\zz J\zz )\zz ,
	\end{align*}
	 where the last equality follows from \eqref{eq:JXpK-1}.
\end{itemize}
In practice, we compute the gradient by means of above two formulations. Note that we do not need to compute $X_\perp$.

Even when a restriction on $X_\perp$ makes $\rgrad{X}$ independent of the choice of $X_\perp$, it is still impacted by the choice of $\rho$. The influence of $\rho$ on the algorithmic performance will be investigated in  \cref{subsection:default setting}.

As we have seen, in the particular case $p=n$ and $\rho=1$, the matrix $X_\perp$ disappears and the canonical-like metric $g_\rho$ reduces to the left-invariant metric used in~\cite{wang2018riemannian,birtea2018optimization}. In that particular case, the Riemannian gradient~\eqref{eq:rgrad-1} is thus also the same {as that} developed in \cite{wang2018riemannian,birtea2018optimization}.

\subsection{Quasi-geodesics}\label{subsec:geo}
Searching along geodesics can be viewed as the most ``geometrically natural'' way to minimize a function on a manifold. In this section, we observe that it is difficult, perhaps impossible, to obtain a closed-form expression of the geodesics of $\symplectic$ with respect to $g_\rho$ in~\eqref{eq:rmetric}. Instead, we develop an approximation, termed quasi-geodesics, defined by imposing that its \emph{Euclidean} second derivative (in lieu of the Riemannian second derivative in $\Rnnpp$ endowed with $g_\rho$) belongs to the normal space~\eqref{eq:normal}.

The geodesics on $(\symplectic,g_\rho)$ are the smooth curves on $\symplectic$ with zero Riemannian acceleration. Since $(\symplectic,g_\rho)$ is a Riemannian submanifold of $(\Rnnpp,g_\rho)$, it follows that the Riemannian acceleration of a curve on $(\symplectic,g_\rho)$ is the orthogonal projection (in the sense of $g_\rho$) onto $\TX$ of its Riemannian acceleration as a curve in $(\Rnnpp,g_\rho)$; see, e.g.,~\cite[Prop.~5.3.2 and (5.23)]{absil2009optimization}.  In other words, the geodesics on $(\symplectic,g_\rho)$ are the curves $Y(t)$ in $\symplectic$ that satisfy
\begin{equation}\label{eq:geo-normal-true}
	\frac{\mathcal{D}}{\mathrm{d}t}\dot{Y}(t)+Y(t)J\varOmega=0 \quad\mbox{with~} \varOmega\in\skewset(2p),
\end{equation}
where we have used the characterization~\eqref{eq:normal} of the normal space, and where $\frac{\mathcal{D}}{\mathrm{d}t}$ denotes the Riemannian covariant derivative \cite[(5.23)]{absil2009optimization} on $\dkh{\Rnnpp,g_\rho}$.

Finding a closed-form solution of~\eqref{eq:geo-normal-true} is elusive. Instead, we replace $\frac{\mathcal{D}}{\mathrm{d}t}\dot{Y}$ in~\eqref{eq:geo-normal-true} by the classical (Euclidean) second derivative $\ddot{Y}$ in $\Rnnpp$, yielding
\begin{equation}\label{eq:geo-normal}
\ddot{Y}(t)+Y(t)J\varOmega=0 \quad\mbox{with~} \varOmega\in\skewset(2p),
\end{equation}
which we term \emph{quasi-geodesic} equation.

The purpose of the rest of this section is to obtain a closed-form solution of~\eqref{eq:geo-normal}.

For ease of notation, we {write $Y$ instead of $Y(t)$ if there is no confusion}. Since the curve $Y(t)$ remains on $\symplectic$, we have $Y(t)\zz J Y(t)=J$ for all $t$. Differentiating this equation
twice with respect to $t$, we have
\begin{equation}\label{eq:geo-differential}
	Y\zz J \dot{Y} + \dot{Y}\zz J Y=0,\quad Y\zz J \ddot{Y}+2 \dot{Y}\zz J \dot{Y} +\ddot{Y}\zz J Y=0.
\end{equation}
Substituting $\ddot{Y}$ of \eqref{eq:geo-normal} into the second equation of \eqref{eq:geo-differential}, we arrive at $\varOmega=-\dot{Y}\zz J \dot{Y}$. Therefore, \eqref{eq:geo-normal} can equivalently be written as 
\begin{equation}\label{eq:geo-ODE}
	\ddot{Y}-YJ(\dot{Y}\zz J \dot{Y})=0.
\end{equation}
Following an approach similar to the one used in~\cite{edelman1998geometry} to obtain the geodesics on the Stiefel manifold, we rewrite the differential equation \eqref{eq:geo-ODE} by involving three types of {first integrals}. Let 
$$C=Y\zz J Y,\quad W=Y\zz J \dot{Y},\quad S=\dot{Y}\zz J \dot{Y}.$$
According to \eqref{eq:geo-ODE}, it is straightforward to verify that
$$\dot{C} = W-W\zz,\quad \dot{W} = S+CJS, \quad \dot{S} = SJW-W\zz JS.$$
Since $Y\in\symplectic$, we have that $C = J$ and therefore $\dot{C}=0$. This in turn implies that $W=W\zz$ and $\dot{W}=0$. {As a~consequence, we get $W=W(0)$.}
Thus, integrals of the above motion are
\begin{align}\label{eq:geo-integral}
{C(t)} = J,\quad {W(t)} = W(0), \quad {S(t)} = e^{(JW)\zz t} S(0) e^{JWt}.
\end{align}

Since $W$ is symmetric, the matrix $JW$ is a Hamiltonian matrix. We will make use of the following classical result stating that the exponential of every Hamiltonian matrix is symplectic.
\begin{proposition}\label{proposition:exponential}
	For every symmetric matrix $W\in\R^{2p\times 2p}$, the matrix exponential $e^{JW}$ is symplectic.
\end{proposition}
\begin{proof}
A proof can be found, e.g., in~\cite[Corollary~3.1]{meyer1992linear} and~\cite[Proposition~2.15]{deGosson2006}. Briefly, letting $U(t) = (e^{tJW})\zz J e^{tJW}$, it follows from $\frac{\mathrm{d}}{\mathrm{d}t} e^{At} = Ae^{At} = e^{At}A$ that $\dot{U}(t) = (e^{tJW})\zz ((JW)\zz J+ J(JW)) e^{tJW} = (e^{tJW})\zz (W\zz-W) e^{tJW} = 0$ for all $t$. Hence $U(t)$ is constant, and since $U(0)=J$, the claim follows. 
\end{proof}

Next, we reorganize \eqref{eq:geo-integral}
as an integrable equation  to derive the closed-form expression of quasi-geodesics. 
\begin{proposition}\label{proposition:geo}
Let $Y(t)$
on $\symplectic$ satisfy the quasi-geodesic equation~\eqref{eq:geo-ODE} with the initial conditions $Y(0)=X\in\symplectic$ and $\dot{Y}(0)=Z\in\TX$. Then 
	\begin{equation}\label{eq:geo}
		Y(t) = Y_X^{\qgeo}(t;Z):=\fkh{X, Z} \exp\dkh{t\begin{bmatrix}
			-JW & JZ\zz J Z \\ I_{2p} & -JW
			\end{bmatrix}} 
		\begin{bmatrix}
		I_{2p} \\ 0
		\end{bmatrix} 
		e^{JWt},
	\end{equation}
	where $W=X\zz J Z$.
\end{proposition}
\begin{proof}
Using the notation and results introduced above, we have
	\begin{align*}
		\arraycolsep=2pt
		\begin{array}{rcl}
		\frac{\mathrm{d}}{\mathrm{d}t}\dkh{Y(t)e^{-JWt}} & = &\dot{Y}{(t)} e^{-JWt} - Y{(t)}e^{-JWt} JW, \\
		\frac{\mathrm{d}}{\mathrm{d}t}\dkh{\dot{Y}(t)e^{-JWt}} & = & \ddot{Y}{(t)} e^{-JWt} - \dot{Y}{(t)}e^{-JWt} JW\\
		&=& Y{(t)}JS{(t)}e^{-JWt} - \dot{Y}(t)e^{-JWt} JW\\
		&=& Y{(t)}J e^{(JW)\zz t} S(0) e^{JW t} e^{-JWt}  - \dot{Y}(t)e^{-JWt} JW \\
		&=& Y{(t)} e^{-JW t} JS(0)- \dot{Y}(t)e^{-JWt} JW,
		\end{array}
	\end{align*}
where the last equality uses the identity $J e^{(JW)\zz t}=e^{-JW t} J$ follows from the fact (\cref{proposition:exponential}) that $e^{JWt}$ is symplectic. Stacking the above equations together yields
	\begin{equation*}
		\frac{\mathrm{d}}{\mathrm{d}t} \fkh{Y(t)e^{-JWt}, \dot{Y}(t)e^{-JWt}} = \fkh{Y(t)e^{-JWt}, \dot{Y}(t)e^{-JWt}} \begin{bmatrix}
		-JW & JS(0) \\ I_{2p} & -JW
		\end{bmatrix},
	\end{equation*}
	which implies that 
	\begin{equation*}
		Y(t)=\fkh{Y(0), \dot{Y}(0)}\exp\dkh{t\begin{bmatrix}
			-JW & JS(0) \\ I_{2p} & -JW
			\end{bmatrix}} 
	\begin{bmatrix}
		I_{2p} \\ 0
	\end{bmatrix} 
	e^{JWt}.
	\end{equation*}
	Since $Y(0)=X$, $\dot{Y}(0)=Z$, $W=W(0)=X\zz J Z$, $S(0)=Z\zz J Z$, we arrive at \eqref{eq:geo}.
\end{proof}

Note that if a tangent vector is {given} by $Z=XJW_Z+JX_\perp K_Z$ {with} $W_Z\in\symset(2p)$ and $K_Z\in\R^{(2n-2p)\times 2p}$, then $W=-W_Z$ in \eqref{eq:geo}. 

In the case of {the} symplectic group ($p=n$), quasi-geodesics can be computed directly by solving $W=Y\zz J \dot{Y}=W(0)$, which leads to 
$Y_U^{\qgeo}(t;Z)=Ue^{-JWt},$
where $Y(0)=U\in\symplecticgroup$ and $W=U\zz J Z$. It turns out that the quasi-geodesic curves for $p=n$ coincide with the geodesic curves corresponding to the indefinite Khvedelidze--Mladenov metric $\jkh{Z_1, Z_2}_U :=\tr\dkh{U\inv Z_1 U\inv Z_2}$; see~\cite[Theorem~2.4]{fiori2011solving}.

\section{Optimization on the symplectic Stiefel manifold}\label{sec:optimization}
With the geometric tools of \cref{sec:geometry} and \cref{section:canonical} at hand, almost everything is in place to apply to problem~\eqref{prob:original} first-order Riemannian optimization methods such as those proposed and analyzed in~\cite{absil2009optimization,RinWir2012,iannazzo2018riemannian,HuaAbsGal2017,BouAbsCar2018,hu2019brief}. The remaining task is to choose a mechanism which, given a search direction in the tangent space $\TX$, returns a curve on $\symplectic$ along that direction. In Riemannian optimization, it is customary to provide such a mechanism by choosing a retraction. In \cref{subsec:quasi} and \cref{subsec:cayley}, we propose two retractions on $\symplectic$. Then, in \cref{subsection:non-monotone}, we work out a non-monotone gradient method for problem~\eqref{prob:original}.

First we briefly recall the concept of (first-order) retraction on a manifold $\manifold$; see~\cite{ADM2002} or~\cite[\S 4.1]{absil2009optimization} for details. Let $\mathrm{T}\manifold:=\bigcup_{X\in\manifold}\mathrm{T}_X \manifold$ be the tangent bundle to $\manifold$. A smooth mapping ${\cal R}: \mathrm{T}\manifold \rightarrow \manifold$ is called a \emph{retraction} 
if the following properties hold for all $X\in\manifold$:
\begin{itemize}
	\item[1)] ${\cal R}_X(0_X)=X$, where $0_X$ is the origin of $\mathrm{T}_X \manifold$;
	\item[2)] $\frac{\text{d}}{\text{dt}}{\cal R}_X(tZ)|_{t=0}=Z$ for all $Z\in \mathrm{T}_X \manifold$,
\end{itemize}
where ${\cal R}_X$ denotes the restriction of ${\cal R}$ to $\mathrm{T}_X \manifold$. Then, given a search direction $Z$ at a point $X$, the above-mentioned mechanism simply returns the curve $t\mapsto {\cal R}_X(tZ)$. 

In most analyses available in the literature, 
the retraction ${\cal R}$ is assumed to be globally defined, i.e., 
defined everywhere on the tangent bundle $\mathrm{T}\manifold$. In~\cite{BouAbsCar2018}, however, ${\cal R}_X$ is only required to be defined locally, in a closed ball of radius $\varrho(X)>0$ centered at $0_X\in \mathrm{T}_X\manifold$, provided that $\inf_k \varrho(X^k)>0$ where the $X^k$'s denote the iterates of the considered method. Henceforth, unless otherwise stated, we only assume that, for all $X\in\manifold$, ${\cal R}_X$ is defined in a neighborhood of $0_X$ in $\mathrm{T}_X\manifold$.

\subsection{Quasi-geodesic curve}  \label{subsec:quasi}
In view of the results of \cref{subsec:geo}, we define ${\cal R}^{\qgeo}$ by
\begin{align}\label{eq:geo-exp}
	{\cal R}^{\qgeo}_X(Z) := Y_X^{\qgeo}(1;Z)=\fkh{X, Z} \exp\dkh{\begin{bmatrix}
		-JW &  JZ\zz J Z \\ I_{2p} & -JW
		\end{bmatrix}} 
	\begin{bmatrix}
	I_{2p} \\ 0
	\end{bmatrix} 
	e^{JW}, 
\end{align}
where $Z\in\TX$ and $W=X\zz J Z$. The concept is illustrated in 
\cref{fig:geodesic}. We prove next that ${\cal R}^{\qgeo}$
is a well-defined retraction on $\symplectic$.

\begin{figure}[htbp]
	\centering
	\small
	\vspace{-7mm}
	\begin{tikzpicture}[scale=.5]
	\filldraw[color=cyan!10] (1.5,6) -- (-1,4) -- (9,3) -- (11,5.3) -- (1.5,6);
	\draw [densely dotted] (1.5,6) -- (-1,4) -- (9,3) -- (11,5.3) -- (1.5,6);
	
	\coordinate [label=-170:$X$] (X) at (3,4.8);
	\coordinate  (Y) at (8,2);	
	\coordinate [label=90:$Z$] (Z) at (8,4.3);
	\coordinate [label=left:$\symplectic$] (M) at (5.3,1);
	\coordinate [label=left:$\TX$] (T) at (0,5);
	\coordinate (C) at ($ (X) ! .5 ! (Z) $);
	\coordinate (A) at ($ (X) ! .5 ! (Y) $);
	\draw[->] (C) --  (Z);
	\draw[->] (X) -- (Z);
	
	\node at ($(Y) + (-1.7,0.1)$) {$\mathcal{R}_X^{\qgeo}(tZ)$};
	\node [fill=black,inner sep=.8pt,circle] at (X) {};
	
	\draw (0,2)	.. controls (0.5,5) and (3,5) .. (6,0.5);
	\draw (0,2)	.. controls (1,2.5) and (2,3) .. (3.92,3);
	\draw (6,0.5)	.. controls (7.5,1.5) and (9.5,2) .. (12,2);
	\draw (0.5,3.4)	.. controls (1.7,5.5) and (11,8) .. (12,2);
	
	\draw [dashed] (X)	.. controls ($(A)+(-0.5,1)$) and ($(A)+(0.5,0.7)$) .. (Y);
	
	\end{tikzpicture}
	\caption{{Quasi-geodesic curve}\label{fig:geodesic}}
\end{figure}
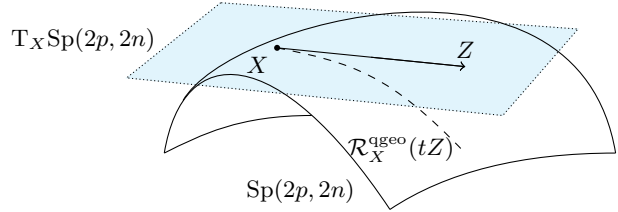

\begin{lemma}\label{lemma:retraction-geo}
	The map ${\cal R}^{\qgeo}: \mathrm{T}\symplectic \to \symplectic$ defined in \eqref{eq:geo-exp} is a globally defined retraction.
\end{lemma}
\begin{proof}
	In view of~\eqref{eq:geo-exp}, ${\cal R}^{\qgeo}$ is well defined and smooth on $\mathrm{T}\symplectic$. From the power series definition of the matrix exponential, we obtain that $Y_X^{\qgeo}(1;0_X) = X$.
	It also follows from this definition (or from the Baker--Campbell--Hausdorff formula) that $\frac{\mathrm{d}}{\mathrm{d}t} \exp(A(t))|_{t=0} = \exp(A(0)) \dot{A}(0)$ when $A(0)$ and $\dot{A}(0)$ commute. This property can be exploited along with the product rule to deduce that $\frac{\mathrm{d}}{\mathrm{d}t} Y_X^{\qgeo}(1;tZ)|_{t=0} = Z$.
\end{proof}



Numerically, computing the exponential will dominate the complexity when $p$ is relatively large; see \cref{fig:geoVSsymplecticCayley} for timing experiments.

\subsection{Symplectic Cayley transform}\label{subsec:cayley}
In this section, we present another retraction, based on the Cayley transform. It follows naturally from the Cayley transform on quadratic Lie groups~\cite[Lemma~8.7]{hairer2006geometric} and the Cayley retraction on the Stiefel manifold~\cite[(7)]{WenYin2013}, with the crucial help of the tangent vector representation given in \cref{cor:ZtoS}.

Given $Q\in\Rnn$, and considering the quadratic Lie group $\mathcal{G}_Q:=\{X\in\Rnn: X\zz QX=Q\}$ and its Lie algebra $\mathfrak{g}_Q:=\{A\in\Rnn: QA+A\zz Q=0\}$, the Cayley transform is given by~\cite[Lemma~8.7]{hairer2006geometric} 
\begin{equation}\label{eq:cay}
	\cay: \mathfrak{g}_Q\to\mathcal{G}_Q: A \mapsto \cay(A) := (I-A)\inv (I+A),
\end{equation}	
which is well defined whenever $I-A$ is invertible.
As for the Cayley retraction on the Stiefel manifold, it is defined, in view of by~\cite[(7)]{WenYin2013}, by $\mathcal{R}_X(Z) = \cay(A_{X,Z}) X$ where $A_{X,Z} = (I-\frac12 XX\zz)ZX\zz - X Z\zz (I-\frac12 XX\zz)$. This inspires the following definition.

\begin{definition}   \label{def:R-cay}
The \emph{Cayley retraction} on the symplectic Stiefel manifold $\symplectic$ is defined, for $X\in\symplectic$ and $Z\in\TX$, by 
\begin{equation}  \label{eq:R-cay}
{\cal R}^{\cay}_X(Z) 
:= \dkh{I-\frac12 S_{X,Z}J}\inv \dkh{I+\frac12 S_{X,Z}J} X,
\end{equation}
where $S_{X,Z}$ is as in \cref{proposition:projection}, i.e., 
$S_{X,Z} = G_XZ(XJ)\zz + XJ (G_XZ)\zz$ and $G_X = I-\frac12 XJX\zz J\zz$.
It is well defined whenever $I-\frac12 S_{X,Z}J$ is invertible.
\end{definition}

In other words, the selected curve along $Z\in\TX$ is 
\begin{equation}\label{eq:cay-Lie}
	Y_X^{\cay}(t;S) := \dkh{I -\frac{t}{2}SJ}\inv\dkh{I +\frac{t}{2}SJ}X = \cay\dkh{\frac{t}{2}SJ}X,
\end{equation}
where $S = S_{X,Z}$ as defined above.

We confirm right away that ${\cal R}^{\cay}$ is indeed a retraction.
\begin{proposition}
The map ${\cal R}^{\cay}: \mathrm{T}\symplectic \to \symplectic$ in~\eqref{eq:R-cay} is a retraction.
\end{proposition}
\begin{proof}
When $Z=0$, we have $S_{X,Z}=0$ and we obtain ${\cal R}^{\cay}_X(Z)=X$, which is the first defining property of retractions. 
For the second property, 
we have $\frac{\mathrm{d}}{\mathrm{d}t} {\cal R}^{\cay}_X(tZ)|_{t=0} = \frac{\mathrm{d}}{\mathrm{d}t} \cay(\frac{t}{2} S_{X,Z}J) X|_{t=0} = \mathrm{D} \cay(\frac{t}{2} S_{X,Z}J)[\frac12 S_{X,Z}J] X|_{t=0} = S_{X,Z}JX = Z$, where the last two equalities come from $\mathrm{D}\cay(A)[\dot{A}] = 2(I-A)^{-1} \dot{A} (I+A)^{-1}$ (see~\cite[Lemma~8.8]{hairer2006geometric}) and \cref{cor:ZtoS}.
\end{proof}

Incidentally, we point out that the Cayley transform \eqref{eq:cay} can be interpreted as the trapezoidal rule for solving ODEs on quadratic Lie groups. According to \cite[IV. (6.3)]{hairer2006geometric}, 
notice that $\mathrm{T}_{X}\mathcal{G}_Q=\{AX: A\in\mathfrak{g}_Q\}$. Hence, the following defines
a differential equation on $\mathcal{G}_Q$:
\begin{equation}\label{eq:ODE-Lie}
	\dot{Y}(t) = A Y(t),\quad Y(0)=X\in\mathcal{G}_Q,
\end{equation}
where $A\in\mathfrak{g}_Q$. To solve it numerically, we adopt one step of the trapezoidal rule over $[0,t]$, 
\begin{equation}\label{eq:trapezoidal}
	Y(t) = X + \frac{t}{2} \dkh{AX+AY(t)}.
\end{equation}
Its solution is given by 
$Y(t)=\cay\dkh{\frac{t}{2}A}X$ whenever $I-\frac{t}{2} A$ is invertible. Since $\cay(\frac{t}{2}A)$, $X\in \mathcal{G}_Q$, it follows that $Y(t)\in\mathcal{G}_Q$. This means that the trapezoidal rule for \eqref{eq:ODE-Lie} is indeed achieved by the Cayley transform \eqref{eq:cay} and remains on $\mathcal{G}_Q$. In particular, letting $A=S_{X,Z}J\in\mathfrak{g}_{\symplecticgroup}$, $Z=AX\in\TX$  and $Y={\cal R}^{\cay}_X(Z)$, the Cayley retraction \eqref{eq:R-cay} can be exactly recovered by the same trapezoidal rule \eqref{eq:trapezoidal}; see \cref{fig:trapezoidal} for an illustration.

\begin{figure}[htbp]
	\small
	\centering
	\vspace{-7mm}
	\begin{tikzpicture}[scale=.5]
	\filldraw[color=cyan!10] (1.5,6) -- (-1,4) -- (9,3) -- (11,5.3) -- (1.5,6);
	\draw [densely dotted] (1.5,6) -- (-1,4) -- (9,3) -- (11,5.3) -- (1.5,6);
	
	\coordinate [label=-170:$X$] (X) at (3,4.8);
	\coordinate [label=-135: ] (Y) at (8,2);	
	\coordinate [label=90:] (Z) at (9,4.1);
	\coordinate [label=left:$\symplectic$] (M) at (5.3,1);
	\coordinate [label=left:$\TX$] (T) at (0,5);
	\coordinate (C) at ($ (X) ! .5 ! (Z) $);
	\coordinate (A) at ($ (X) ! .5 ! (Y) $);
	
	\draw[-{stealth[red]},red,thick] (X) -- (C);
	\draw[->] (C) --  (Z);
	\draw[-{stealth[blue]},blue,thick] (C) -- (Y);
	
	\node at ($(C) + (-0.9,0.6)$) {$\frac{1}{2}AX$};
	\node at ($(Y) + (-0.3,1.5)$) {$\frac{1}{2}AY$};
	\node at ($(Z) + (0.2,0.5)$) {$Z=AX$};
	\node at ($(Y) + (-1.3,-0.3)$) {${\cal R}^{\cay}_X(Z)$};
	\node [fill=black,inner sep=.8pt,circle] at (X) {};
	\node [fill=black,inner sep=.8pt,circle] at (Y) {};
	
	\draw (0,2)	.. controls (0.5,5) and (3,5) .. (6,0.5);
	\draw (0,2)	.. controls (1,2.5) and (2,3) .. (3.92,3);
	\draw (6,0.5)	.. controls (7.5,1.5) and (9.5,2) .. (12,2);
	\draw (0.5,3.4)	.. controls (1.7,5.5) and (11,8) .. (12,2);
	
	\draw [dashed] (X)	.. controls ($(A)+(-0.5,1)$) and ($(A)+(0.5,0.7)$) .. (Y);
	
	\end{tikzpicture}
	\caption{Trapezoidal rule on the symplectic Stiefel manifold\label{fig:trapezoidal}}
\end{figure}
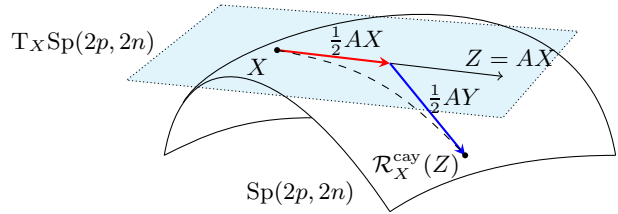

In the case of the symplectic group ($p=n$), it can be shown that ${\cal R}^{\cay}$ reduces to the retraction proposed in~\cite[Proposition~2]{birtea2018optimization}.

The retraction ${\cal R}^{\cay}_X(tZ)$ in~\eqref{eq:R-cay} is not globally defined. It is defined if and only if $I-\frac{t}{2}S_{X,Z}J$ is invertible, i.e., $\frac2t$ is not an eigenvalue of the Hamiltonian matrix $S_{X,Z}J$. Since the eigenvalues of Hamiltonian matrices come in opposite pairs, we have that $Y_X^{\cay}(t;S)$
exists for all $t\geq0$ if and only if $SJ$ has no real eigenvalue. 
This situation constrasts with the Cayley retraction on the (standard) Stiefel manifold, which is everywhere defined due to the fact that $I-A$
is invertible for all skew-symmetric $A$. 

The fact that ${\cal R}^{\cay}$ is not globally defined does not preclude us from applying the convergence and complexity results of~\cite{BouAbsCar2018}. However, we have to ensure that ${\cal R}_X$ is defined in a closed ball of radius $\varrho(X)>0$ centered at $0_X\in \mathrm{T}_X\manifold$, with $\inf_k \varrho(X^k)>0$ where the $X^k$'s denote the iterates of the considered method. An assumption that guarantees this condition is that (i) the objective function $f$ in~\eqref{prob:original} has compact sublevel sets and (ii) the considered optimization scheme guarantees that $f(X^k)\leq f(X^0)$ for all $k$. Indeed, in that case, since $\{X^k\}_{k=0,1,...}$ remains
in a compact subset of $\symplectic$, it is possible to find $\rho$ such that, for all $k$, if $\|Z\|_{X^k}\leq\rho$, then the spectral radius of $\frac12 S_{X^k,Z}J$ is stricly smaller than one, making $I-\frac12 S_{X^k,Z}J$ invertible. 

The main computational cost of ${\cal R}^{\cay}_X(Z)$ (\cref{def:R-cay}) is the symplectic Cayley transform, which requires solving linear systems with the $2n\times 2n$ system matrix $(I-\frac12 S_{X,Z}J)$. In general, {this} has complexity $O(n^3)$. However, as we now show, the low-rank structure of $S_{X,Z}$ can be exploited to reduce the linear system matrix to the size $4p\times 4p$, thereby inducing a considerable reduction of computational cost when $p\ll n$. The development parallels the one given in~\cite[Lemma~4(1)]{WenYin2013} for the (standard) Stiefel manifold.
\begin{proposition}
	Let $S=LR\zz + RL\zz=UV\zz$, where $L,R\in\Rnnpp$ and $U=[L~R]\in\R^{2n\times 4p}$, $V=[R~L]\in\R^{2n\times 4p}$. If $I+\frac{t}{2} V\zz J\zz U$ is invertible, then~\eqref{eq:cay-Lie} admits the expression
	\begin{equation}\label{eq:cayley-simple}
		Y_X^{\cay}(t;S)=X+tU\dkh{I+\frac{t}{2} V\zz J\zz U}\inv V\zz JX.
	\end{equation}
	In particular, if we choose $L=-H_X \nabla \bar{f}(X)$ and $R=XJ$, then we get $S=-S_X$ with $S_X$ given in the gradient formula (\cref{proposition:rgrad}), and we have that ${\cal R}^{\cay}_X(-t\rgrad{X})$ is given by
	\begin{equation}\label{eq:cayley-grad-simple}
	Y_X^{\cay}(t;-S_X)=X+t[-P_{f}~XJ] \dkh{I+\frac{t}{2} 	
	\begin{bmatrix}
			E_\rho & J\zz \\ P_f\zz J\zz P_f & -E\zz_\rho
	\end{bmatrix}	
	}\inv 
	\begin{bmatrix}
			I \\ -E\zz_\rho J
	\end{bmatrix},
	\end{equation}
	where $H_X$ is defined in \cref{proposition:rgrad}, $P_f:=H_X \nabla \bar{f}(X)$, and $E_\rho:=\frac{\rho}{2} X\zz \nabla \bar{f}(X)$.
\end{proposition}
\begin{proof}
	By using Sherman--Morrison--Woodbury (SMW) formula \cite[(2.1.4)]{golub2013matrix}:
	\begin{equation*}\label{eq:SMW}
		(A+\bar{U}\bar{V}\zz)\inv=A\inv-A\inv \bar{U} \dkh{I+\bar{V}\zz A\inv \bar{U}} \inv \bar{V}\zz A\inv
	\end{equation*}
	with $A=I$, $\bar{U}=-\frac{t}{2}U$ and $\bar{V}=J\zz V$, it follows that $\dkh{I -\frac{t}{2}SJ}\inv  =\dkh{I -\frac{t}{2}UV\zz J}\inv=(A+\bar{U}\bar{V}\zz)\inv
	 =I-\frac{t}{2} U(I+\frac{t}{2} V\zz J\zz U)\inv V\zz J\zz.$
	In view of \eqref{eq:cay-Lie}, it turns out that
	\begin{align*}
		\arraycolsep=2pt		
		\begin{array}{rcl}
			Y_X^{\cay}(t;S) &=& \cay\dkh{\frac{t}{2}SJ}X = \dkh{I -\frac{t}{2}SJ}\inv\dkh{I +\frac{t}{2}SJ}X \\
			&=& X+\frac{t}{2}U\dkh{I+(I+\frac{t}{2} V\zz J\zz U)\inv{(I-\frac{t}{2}V\zz J\zz U)}}V\zz JX\\
			&=& X+tU\dkh{I+\frac{t}{2} V\zz J\zz U}\inv V\zz JX,
		\end{array}
	\end{align*}
	which completes the proof of \eqref{eq:cayley-simple}. Substituting $L=-H_X \nabla \bar{f}(X)$ and $R=XJ$ into $Y_X^{\cay}(t;S)$, it is straightforward to arrive at \eqref{eq:cayley-grad-simple}.
\end{proof}

\subsection{A non-monotone line search scheme on manifolds}\label{subsection:non-monotone}
Now, since
we can compute a gradient $\rgradg{X}$
and a retraction ${\cal R}_X(Z)$, various first-order Riemannian optimization methods can be applied to problem~\eqref{prob:original}. Here, we adopt an approach
called non-monotone line search~\cite{zhang2004nonmonotone}. It was extended to the Stiefel manifold in~\cite{WenYin2013} and to general Riemannian manifolds in~\cite{iannazzo2018riemannian} and~\cite[\S 3.3]{hu2019brief}. The non-monotone approach has been observed to work well on various Riemannian optimization problems.
In this section, we first present and analyze the non-monotone line search algorithm on general Riemannian manifolds endowed with a retraction that need not be globally defined. Then we apply the algorithm and its analysis to the case of the symplectic Stiefel manifold $\symplectic$.

Given $\beta\in(0,1)$, a backtracking parameter $\delta\in(0,1)$, a search direction $Z^k$ and a trial step size $\gamma_k>0$, the non-monotone line search procedure proposed in \cite{hu2019brief} uses the step size $t_k = \gamma_k \delta^h$, where $h$ is the smallest integer such that 
\begin{equation}\label{eq:non-monotone}
	f\dkh{{\cal R}_{X^k}(t_k Z^k)} \le c_k + \beta t_k \rmetrick{\rgradg{X^k},Z^k},
\end{equation}
and the next iterate is given by $X^{k+1}={\cal R}_{X^k}(t_k Z^k)$.
The scalar $c_{k}$ in~\eqref{eq:non-monotone} is a convex combination of $c_{k-1}$ and $f(X^k)$. Specifically, we set $q_0=1 , c_0= f(X^0)$,
\begin{align}\label{eq:non-monotone-CQ}
	\begin{array}{l}
		q_k = \alpha q_{k-1} + 1,\\
		c_k = \frac{\alpha q_{k-1}}{q_k} c_{k-1}  + \frac{1}{q_k} f(X^k)
	\end{array}
\end{align}
with a parameter $\alpha\in[0,1]$. 
When $\alpha=0$, it follows that $q_k=1$ and $c_k=f(X^k)$, and the non-monotone condition \eqref{eq:non-monotone} reduces to the standard Armijo backtracking line search
\begin{equation}\label{eq:monotone}
f\dkh{{\cal R}_{X^k}(t_k Z^k)} \le f(X^k) + \beta t_k \rmetrick{\rgradg{X^k},Z^k}.
\end{equation}

The corresponding Riemannian gradient method is presented in \crefalg{alg:non-monotone gradient}. 
The search direction is chosen as the Riemannian antigradient in line~\ref{alg:nmg:grad}.
On the other hand, there is no restriction on the choice of the trial step size $\gamma_k$ in line~\ref{alg:nmg:gamma}. One possible strategy is the Barzilai--Borwein (BB) method \cite{BB}, which often accelerates the convergence of gradient methods when the search space is a Euclidean space.  This method was extended to general Riemannian manifolds {in} \cite{iannazzo2018riemannian}. We implement and compare several choices of $\gamma_k$; see \cref{subsection:default setting} for details.

\begin{algorithm2e}[htbp]
	\caption{Riemannian gradient method with non-monotone line search}
	\label{alg:non-monotone gradient}
	\SetKwInOut{Input}{Input}\SetKwInOut{Output}{Input}
	\SetKwComment{Comment}{}{}
	\BlankLine 
	\textbf{Input:}  $X^0\in\manifold$.\\
	\textbf{Require:}  Continuously differentiable function $f:\manifold\to\R$; retraction $\mathcal{R}$ on $\manifold$ defined on $\dom(\mathcal{R})$;
$\beta, \delta\in(0,1)$, $\alpha \in [0,1]$,
$0<\gamma_\mathrm{min}<\gamma_\mathrm{max}$, 
$c_0 = f(X^0)$, $q_0=1$, $\gamma_0 = f(X^0)$.\\
	\For{$k=0,1,2,\dots$}
	{
		
		Choose $Z^k=-\rgradg{X^k}$.  \label{alg:nmg:grad}
		
	        Choose a trial step size $\gamma_k>0$
		and {set $\gamma_k=\max(\gamma_\mathrm{min},\min(\gamma_k,\gamma_\mathrm{max}))$}. Find the smallest integer $h$ such that $\gamma_k \delta^h Z^k\in\dom(\mathcal{R})$ and the non-monotone condition \eqref{eq:non-monotone} hold. 
	Set $t_k=\gamma_k \delta^h$.  \label{alg:nmg:gamma}
		
		Set $X^{k+1} = {\cal R}_{X^k}(t_k Z^k)$.
		
		Update $q_{k+1}$ and $c_{k+1}$ by \eqref{eq:non-monotone-CQ}.
	}
	\textbf{Output:} Sequence of iterates $\{X^k\}$.
\end{algorithm2e}

Next we prove the convergence of \crefalg{alg:non-monotone gradient} on a general Riemannian manifold $\manifold$.
Note that, in terms of convergence analysis, the only relevant difference between \crefalg{alg:non-monotone gradient} and \cite[Algorithm~1]{hu2019brief} is that the latter assumes that the retraction $\mathcal{R}$ is globally defined, whereas we only assume that, for all $X\in\manifold$, ${\cal R}_X$ is defined in a neighborhood of $0_X$ in $\mathrm{T}_X\manifold$. In other words,
we only assume that, for every $X\in\manifold$, there exists a ball $\mathcal{B}^r_{0_{X}}:=\{Z\in\mathrm{T}_X\manifold: \norm{Z}_{X}< r\}\subseteq\dom(\mathcal{R})$.
Hence, the convergence result in \cite[Theorem~3.3]{hu2019brief} does not directly apply to our case.


First, 
we show that \crefalg{alg:non-monotone gradient} does not abort.
\begin{proposition}\label{proposition:infinite-sequence}
	\crefalg{alg:non-monotone gradient} generates an infinite sequence of iterates.
\end{proposition} 
\begin{proof}
        Let $X^k$ be the current iterate.
	In view of $Z^k=-\rgradg{X^k}$, by applying \cite[Lemma~1.1]{zhang2004nonmonotone} to the Riemannian case, it yields $f(X^k)\leq c_k$.  Since $\mathcal{R}$ is locally defined and $\delta,\beta\in(0,1)$, it follows that
	\begin{align*}
	\lim\limits_{h\to +\infty}\frac{f\dkh{{\cal R}_{X^k}(\gamma_k \delta^h Z^k)}-c_k}{\gamma_k \delta^h}&\leq\lim\limits_{h\to +\infty}\frac{f\dkh{{\cal R}_{X^k}(\gamma_k \delta^h Z^k)}-f(X^k)}{\gamma_k \delta^h}\\
	&=\rmetrick{\rgradg{X^k},Z^k}
	<\beta\rmetrick{\rgradg{X^k},Z^k}.
	\end{align*}
	It implies that there exists $\bar{h}\in\mathbb{N}$ such that $t_k=\gamma_k\delta^{\bar{h}}\in(0,\frac{r}{\|Z^k\|_{X^k}})$
and the non-monotone condition \eqref{eq:non-monotone} hold. It means that $t_k Z^k\in\mathcal{B}^r_{0_{X^k}}\subseteq\dom(\mathcal{R})$, and hence it is accepted.
\end{proof}

Next, we give the proof of the convergence for \crefalg{alg:non-monotone gradient}.

\begin{theorem}\label{theorem:convergence}
	Let $\{X^k\}$ be an infinite sequence of iterates generated by \crefalg{alg:non-monotone gradient}. Then every accumulation point $X^*$ of $\{X^k\}$ such that $0_{X^*}\in\mathrm{T}_{X^*}\manifold$ in the interior of $\dom(\mathcal{R})$,  is a critical point of $f$, i.e., $\rgradg{X^*}=0$.
\end{theorem}
\begin{proof}
	We adapt the proof strategy of \cite[Theorem~3.3]{hu2019brief} (that invokes~\cite[Theorem~4.3.1]{absil2009optimization}, which itself is a generalization to manifolds of the proof of~\cite[Proposition~1.2.1]{Ber95}) in order to handle the locally defined retraction. The adaptation consists in defining and making use of the neighborhood $\mathcal{N}_{0_{X^*}}$. 

	Since $X^*$ is an accumulation point of $\{X^k\}$, there exists a subsequence $\{X^k\}_{k\in\mathcal{K}}$ that converges to it. 
	In view of \eqref{eq:non-monotone} and \eqref{eq:non-monotone-CQ}, it holds
	that $q_{k+1}=1+\alpha q_k=1+\alpha+\alpha^2q_{k-1}=\dots=\sum_{i=0}^{k+1}\alpha^i
	 \le k+2
	 $.
	 Moreover, it follows that
	\begin{align*}
	c_{k+1}-c_k&=\frac{\alpha q_{k}c_{k}+f(X^{k+1})}{q_{k+1}} -c_k=\frac{\alpha q_{k}c_{k}+f(X^{k+1})}{q_{k+1}} -\frac{\alpha q_{k}+1}{q_{k+1}}c_k\\
	&=\frac{f(X^{k+1})-c_k}{q_{k+1}} \leq -\frac{\beta t_k\norm{\rgradg{X^k}}^2_{X^k}}{q_{k+1}}<0.
	\end{align*}
	Hence, $\{c_k\}$ is monotonically decreasing. Since $f(X^k)\leq c_k$ (in the proof of \cref{proposition:infinite-sequence}) and $\{f(X^k)\}_{k\in\mathcal{K}}\to f(X^*)$, it readily follows that $ c_\infty:=\lim_{k\to+\infty}c_k>-\infty$.
	Summing above inequalities and using $q_{k+1}\le k+2$, it turns out
	$$\sum_{k=0}^{\infty}\frac{\beta t_k\norm{\rgradg{X^k}}^2_{X^k}}{k+2}
	\leq \sum_{k=0}^{\infty}\frac{\beta t_k\norm{\rgradg{X^k}}^2_{X^k}}{q_{k+1}}
	\leq
	\sum_{k=0}^{\infty}(c_k-c_{k+1})
	=c_0-c_\infty<\infty.
	$$
	This inequality implies that
	\begin{equation}\label{eq:convergence-1}
	\lim\limits_{k\to+\infty} t_k\norm{\rgradg{X^k}}^2_{X^k}=0, \quad k\in\mathcal{K}.
	\end{equation}
	
	For the sake of contradiction, suppose that $X^*$ is not a critical point of $f$, i.e., that $\rgradg{X^*}\neq0$.  
It readily follows from \eqref{eq:convergence-1} that $\{t_k\}_{k\in\mathcal{K}}\to0$. By the construction of \crefalg{alg:non-monotone gradient}, the step size $t_k$ has the form of $t_k=\gamma_k \delta^h$ with $\gamma_k\ge\gamma_{\min}>0$.
Since $0_{X^*}$ is in the interior of $\dom(\mathcal{R})$, there exists a neighborhood $\mathcal{N}_{0_{X^*}}:=\{(X,Z)\in\mathrm{T}\manifold: \mathrm{dist}^2(X,X^*)+ \norm{Z}_{X}^2 < r\}\subseteq\dom(\mathcal{R})$.
Due to $\{t_k\}_{k\in\mathcal{K}}\to0$, it follows that there exists $\bar{k}\in\mathcal{K}$ and $h_k\in\mathbb{N}$ with $h_k\geq1$ such that for all $k>\bar{k}$, $t_k=\gamma_k \delta^{h_k}$ and $(X^k,\frac{t_k}{\delta} Z^k)\in\mathcal{N}_{0_{X^*}}$.
Since $h_k$ is the smallest integer such that the non-monotone condition \eqref{eq:non-monotone} holds, it follows that $\frac{t_k}{\delta} Z^k$ does not satisfy \eqref{eq:non-monotone}, i.e.,
	\begin{align}\label{eq:non-monotone-contradiction}
	f(X^k) - f\dkh{{\cal R}_{X^k}(\frac{t_k}{\delta} Z^k)} 
	\leq
	c_k - f\dkh{{\cal R}_{X^k}(\frac{t_k}{\delta} Z^k)}
	< \beta \frac{t_k}{\delta}  \norm{\rgradg{X^k}}^2_{X^k}.
	\end{align}
    The mean value theorem for \eqref{eq:non-monotone-contradiction} ensures that there exists $\bar{t}_k\in\fkh{0,\frac{t_k}{\delta}}$ such that 
	\[
\mathrm{D}(f\circ {\cal R}_{X^k})(\bar{t}_k Z^k)[Z^k]
<\beta \norm{\rgradg{X^k}}^2_{X^k},\quad \mbox{for all~} {k}\in\mathcal{K}, k>\bar{k}.
\]
We now take the limit in the above inequality as $k\to\infty$ 
over $\mathcal{K}$. 
Using the fact that $\mathrm{D}(f\circ {\cal R}_{X^*})(0_{X^*}) = \mathrm{D}f(X^*)$ in view of the defining properties of a retraction, we obtain $\norm{\rgradg{X^*}}^2_{X^*}\leq\beta \norm{\rgradg{X^*}}^2_{X^*}$. Since $\beta<1$, this is a contradiction with the supposition that $\rgradg{X^*}\neq0$.
\end{proof}

\crefalg{alg:non-monotone gradient} applies to $\manifold = \symplectic$ as follows. Pick a constant $\rho>0$ and one of the orthonormalization conditions (I) or (II) for $X_\perp$ in order to make~\eqref{eq:rmetric} a bona-fide Riemannian metric. In line~\ref{alg:nmg:grad}, the gradient is then as stated in~\cref{proposition:rgrad}. Finally, choose the retraction $\mathcal{R}$ as either the quasi-geodesic curve~\eqref{eq:geo-exp} or the symplectic Cayley transform~\eqref{eq:R-cay}.
The initialization techniques for $X^0\in\symplectic$ 
and the stopping criterion
will be discussed in \cref{subsection:Stopping criteria}. 

\begin{corollary}\label{theorem:convergence-symplectic}
	Apply \crefalg{alg:non-monotone gradient} to $\symplectic$ as specified in the previous paragraph. 
	 Let $\{X^k\}$ be an infinite sequence of iterates generated by \crefalg{alg:non-monotone gradient}. Then every accumulation point $X^*$ of $\{X^k\}$ is a critical point of $f$, i.e., $\rgradg{X^*}=0$.
\end{corollary}
\begin{proof}
	First, the quasi-geodesic retraction is globally defined. Hence, \cref{theorem:convergence} can be directly applied. 
	
Next, we consider the Cayley-based algorithm. 
In order to conclude by invoking \cref{theorem:convergence}, it is sufficient to show that, for all $X\in\symplectic$, $0_X\in\TX$ is in the interior of $\dom({\cal R}^{\cay})$. In view of the properties of retractions, we have that $0_X\in\dom(\mathcal{R}^{\cay})$, and we finish the proof by showing that all points of $\dom(\mathcal{R}^{\cay})$ are in its interior, namely, $\dom(\mathcal{R}^{\cay})$ is open. To this end, let $F:\mathrm{T}\symplectic\to\R:(X,Z)\mapsto\det(I-\frac12 S_{X,Z})$, where $S_{X,Z}$ is as in \cref{def:R-cay}. Since $F$ is continuous and $\{0\}$ is closed, it follows that $F\inv(0)=\{(X,Z)\in\mathrm{T}\symplectic: F(X,Z)=0\}$ is a closed set. Its complement in the tangent bundle is thus open, and it is also the domain of $\mathcal{R}^{\cay}$.
\end{proof}

\section{Numerical experiments}\label{sec:Numerical Experiment}
In this section, we report the numerical performance of \crefalg{alg:non-monotone gradient}. Both methods based on the {quasi-geodesics} \eqref{eq:geo-exp}  and the symplectic Cayley transform \eqref{eq:R-cay}
are evaluated. We first introduce implementation details in \cref{subsection:Stopping criteria}. To determine default settings, we investigate the parameters of our proposed algorithms in \cref{subsection:default setting}. Finally, the efficiency of \crefalg{alg:non-monotone gradient} is assessed
by solving several different problems. The 
experiments are performed on a workstation with two Intel(R) Xeon(R) Processors Silver 4110 (at 2.10GHz$\times 8$, 12M Cache) and 384GB of RAM running 
MATLAB R2018a under Ubuntu 18.10. The code that produced the results is available from \href{https://github.com/opt-gaobin/spopt}{https://github.com/opt-gaobin/spopt}.

\subsection{Implementation details}\label{subsection:Stopping criteria}
As we mentioned in \cref{subsubsec:rgrad}, we propose two strategies to compute the Riemannian gradient $\rgrad{X}$. Both algorithms with type (I) and type (II) perform well in our preliminary experiments. In this section, \crefalg{alg:non-monotone gradient} with the symplectic Cayley transform is denoted by ``Sp-Cayley", and its instances using type (I) and (II) for the gradient are represented as ``Sp-Cayley-I" and ``Sp-Cayley-II", respectively. 
Similarly,  quasi-geodesic algorithms  are denoted by ``Quasi-geodesics", ``Quasi-geodesics-I" and ``Quasi-geodesics-II".

We adopt formula \eqref{eq:rgrad-2}, i.e., $H_X\nabla\bar{f}(X) (XJ)\zz JX +XJ (H_X\nabla\bar{f}(X))\zz JX$, to assemble $\rgrad{X}$ for all the algorithms.  In order to save flops and obtain a good feasibility (see \cref{fig:geoVSsymplecticCayley}), we choose \eqref{eq:cayley-grad-simple} to compute the Cayley retraction.  Note that we keep the calculation of $(XJ)\zz JX$ in the first term of $\rgrad{X}$, although it is trivial that $(XJ)\zz JX=I$ for $X\in\symplectic$. This is due to our observation that the feasibility of Quasi-geodesics gradually degrades
when we omit this calculation. 

At the beginning of \crefalg{alg:non-monotone gradient}, we need a feasible point $X^0\in\symplectic$ to start the iteration. The easiest way to generate a symplectic matrix is to choose the ``identity" matrix in $\symplectic$, namely, $I^0=\fkh{\begin{smallmatrix} 
	I_p & 0 & 0 & 0 \\ 0& 0 & I_p & 0
\end{smallmatrix}}\zz$. Moreover, by using \cref{proposition:exponential}, we suggest the following strategies to generate an initial point:
\begin{itemize}
	\item[1)] $X^0=I^0$;
	\item[2)] $X^0=I^0 e^{J(W+W\zz)}$, where $W$ is randomly generated by \texttt{W=randn(2*p,2*p)};
	\item[3)] $X^0$ is assembled by the first $p$ columns and $(n+1)$-th to $(n+p)$-th columns of $e^{J(W+W\zz)}$, where $W$ is randomly generated by \texttt{W=randn(2*n,2*n)}.
\end{itemize}
The matrix exponential is computed by the function $\texttt{expm}$ in MATLAB. Unless otherwise specified, we 
choose strategy 2) as our initialization.

For a stopping criterion, we check the following two conditions:
\begin{eqnarray}\label{eq:stop}
&\norm{\rgrad{X^k}}\ff\le\epsilon,&\\
\label{eq:stop1}
&\frac{\norm{X^k-X^{k+1}}\ff}{\sqrt{2n}} < \epsilon_x \quad  \mbox{and}\quad
\frac{\abs{f(X^k)-f(X^{k+1})}}{\abs{f(X^k)}+1} < \epsilon_f  &
\end{eqnarray}
with given tolerances $\epsilon,\epsilon_x,  \epsilon_f>0$. 
{We terminate the algorithms {once} one of the criteria \eqref{eq:stop}-\eqref{eq:stop1}} or
a maximum iteration number $\mathrm{MaxIter}$ is reached. The default tolerance parameters are chosen as 
$\epsilon = 10^{-5}$, $\epsilon_x = 10^{-5}$, $\epsilon_f = 10^{-8}$ and $\mathrm{MaxIter}=1000$. 
For parameters to control the non-monotone line search, we follow the choices in the code OptM\footnote{Available from \href{https://github.com/wenstone/OptM}{https://github.com/wenstone/OptM}.} \cite{WenYin2013}, specifically, $\beta=10^{-4}, \delta=0.1, \alpha=0.85$, as our default settings. In addition, we choose $\gamma_\mathrm{min}=10^{-15}$, $\gamma_\mathrm{max}=10^{15}$, 
and a trial step size $\gamma_k$ as in \eqref{eq:ABB}.

\subsection{Default settings of the algorithms}\label{subsection:default setting}
In this section, we study the performance and robustness of our algorithms by choosing different parameters. All the comparisons and results are based on a test problem, called the nearest symplectic matrix problem, which aims to calculate the nearest symplectic matrix to a~target matrix $A\in\Rnnpp$ with respect to the Frobenius norm, i.e., 
\begin{equation}\label{eq:nearest}
\min\limits_{X\in\symplectic} \norm{X-A}\fs.
\end{equation} 
{The special case of this problem on the symplectic group (i.e., $p=n$) was studied in~\cite{wu2010critical}}. 
In our experiments, $A$ is randomly generated by \texttt{A=randn(2*n,2*p)}, then it is scaled using \texttt{A=A/norm(A)}. 

\subsubsection{A comparison between quasi-geodesics and symplectic Cayley transform}\label{subsection:geoVScayley}
In \crefalg{alg:non-monotone gradient}, we can choose between two different retractions: quasi-geodesic and Cayley. 
In order to investigate the numerical performance of the two alternatives, we solve the nearest symplectic matrix problem \eqref{eq:nearest} of size $2000 \times 400$
and choose the parameter for the metric $g_{\rho}$ to be $\rho=1$.  We stop our algorithms only when $\mathrm{MaxIter}=120$ is reached. The evolution of the norm of the gradient and the feasibility violation $\norm{X\zz JX-J}\ff$ for both algorithms is shown in \cref{fig:geoVSsymplecticCayley}. 
It illustrates that Sp-Cayley performs better than Quasi-geodesics in terms of efficiency and feasibility.
Therefore, we choose Sp-Cayley as our default algorithm and the following experiments will focus on Sp-Cayley. 

\begin{figure}[htpb]
	\centering
	\subfigure[F-norm of Riemannian gradient]
	{\includegraphics[scale=.34]
		{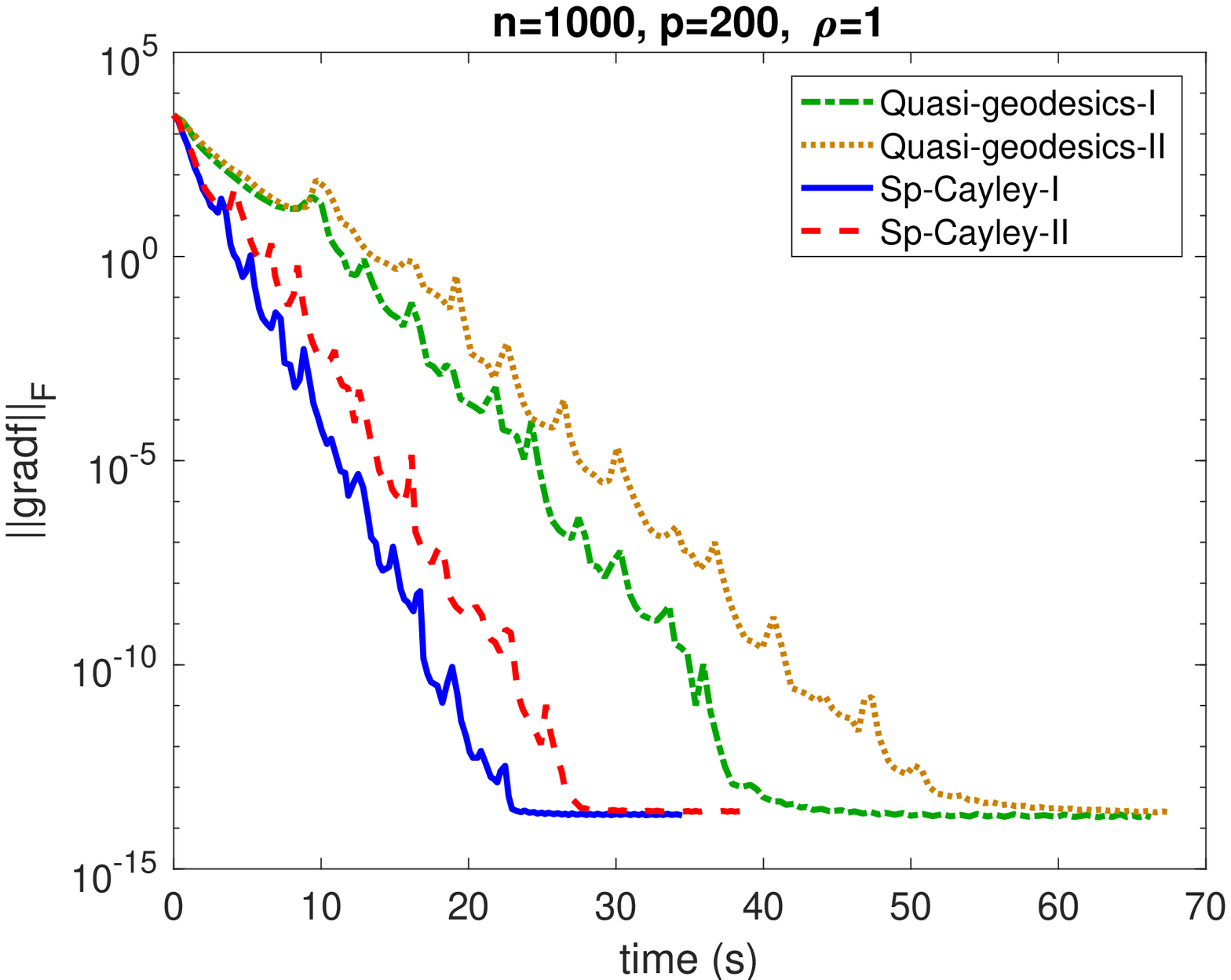}}
	\quad
	\subfigure[Feasibility violation]
	{\includegraphics[scale=.34]
		{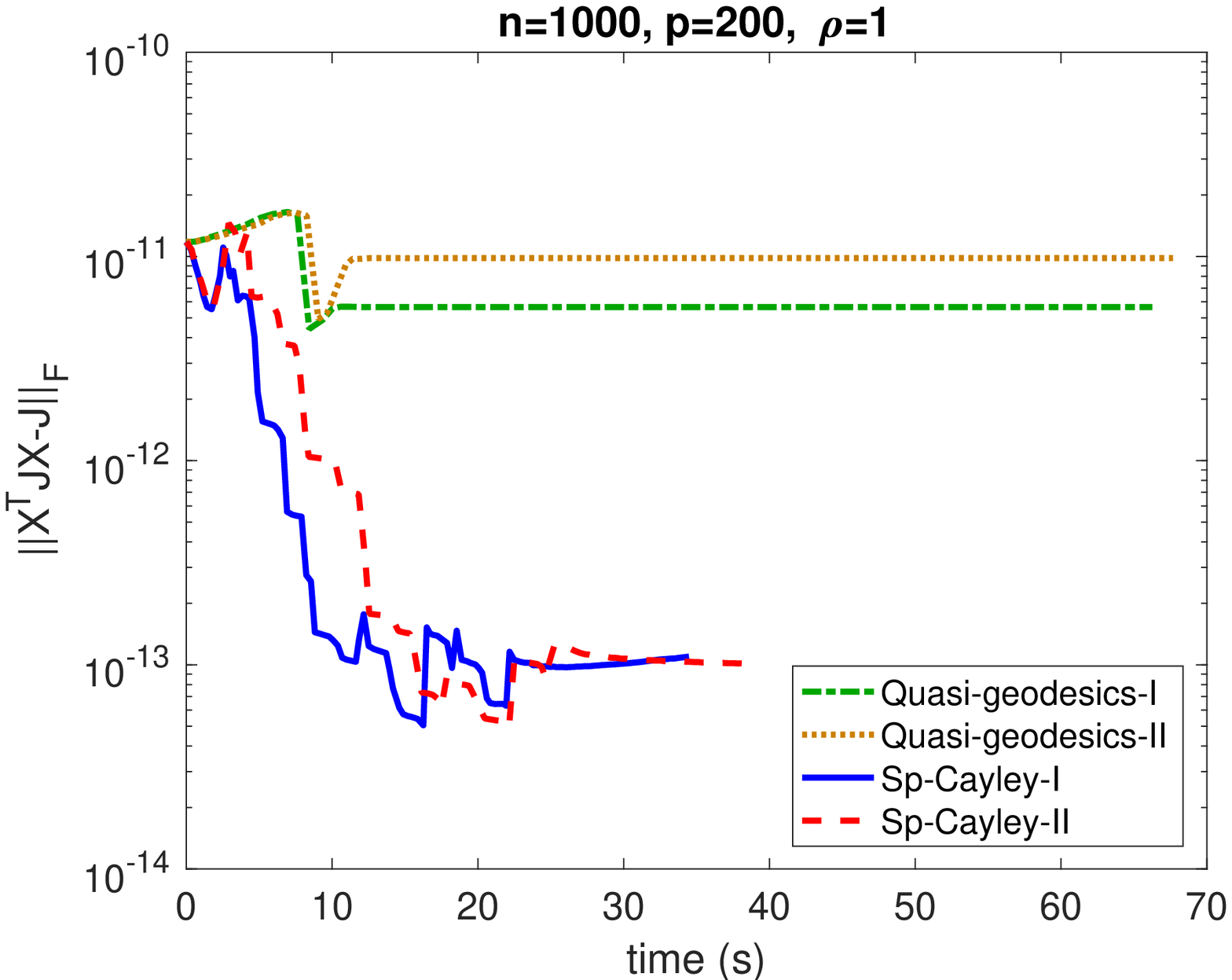}}
	\caption{A comparison of \crefalg{alg:non-monotone gradient} {based on the} quasi-geodesic {curve} and {the} symplectic Cayley transform\label{fig:geoVSsymplecticCayley}}
\end{figure}

\subsubsection{A comparison of different metrics}  \label{sec:exper-metrics}
In this section, we compare the performance of Sp-Cayley with different metrics.
Namely, we compare Sp-Cayley-I and Sp-Cayley-II for a~set of parameters $\rho=2^l$ with $l$ chosen from $\{-3,-2,-1,0,1,2,3\}$.  We run {both} algorithms 100 times on randomly generated nearest symplectic matrix problems of size $2000\times 40$. Note that for each instance, Sp-Cayley-I and Sp-Cayley-II use the same initial guess. In order to get an average performance, we stop algorithms only when $\norm{\rgrad{X^k}}\ff\le10^{-4}$.  A~summary of numerical results is reported in \cref{fig:metric}. It displays average iteration numbers and the feasibility violation for the two algorithms with different $\rho$. We can learn from {the} figures that:  
\begin{itemize}
	\item The value $\rho^*$ at which {Sp}-Cayley-I and Sp-Cayley-II achieve the best performance is approximately equal to $1/2$ and 1, respectively. The difference of $\rho^*$ is due to the different choice of $X_\perp$, which also has an effect on the metric $g_\rho$.  We have observed that $\rho^*$ may vary {for} different objective functions. This indicates that tuning $\rho$ for the problem class of interest may significantly improve the performance of \crefalg{alg:non-monotone gradient}.
	
	\item Sp-Cayley-I has a lower average iteration number than Sp-Cayley-II when $\rho\le 1$. Over all the variants considered in \cref{fig:metric}, Sp-Cayley-I with $\rho=1/2$ has the best performance.

	\item Both Sp-Cayley-I and Sp-Cayley-II show a loss of feasibility when $\rho$ becomes large. A possible reason is that the non-normalized second term of $H_X=JX_\perp X\zz_\perp J\zz +\frac{\rho}{2}XX\zz$ in the Riemannian gradient becomes dominant. 
\end{itemize}
According to the above observations,  we choose $\rho=1/2$ for Sp-Cayley-I and $\rho=1$ for Sp-Cayley-II as our default settings.

\begin{figure}[htpb]
	\centering
	\subfigure[Average iteration number]
	{\includegraphics[scale=.34]
		{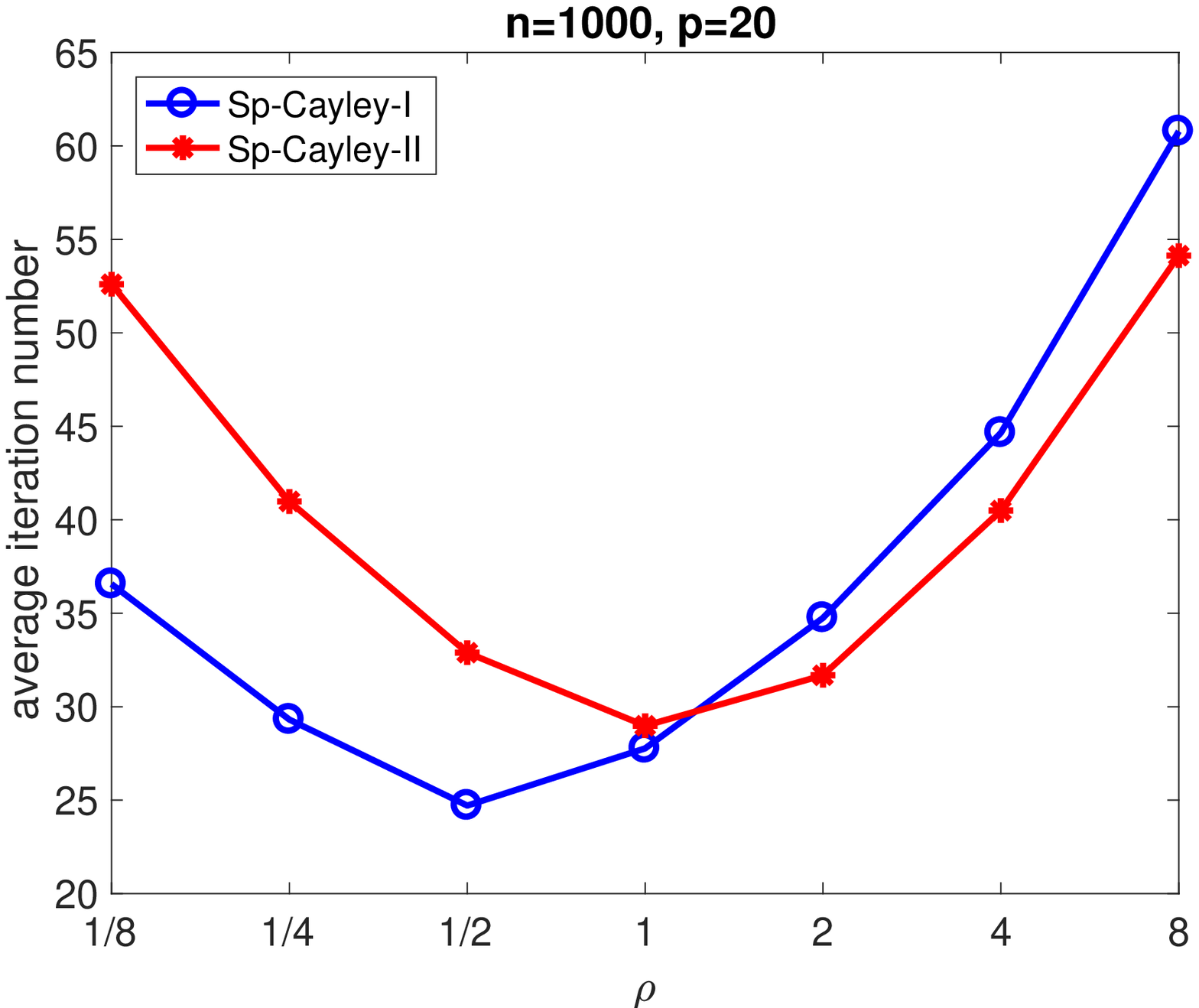}}
	\quad
	\subfigure[Average feasibility violation]
	{\includegraphics[scale=.34]
		{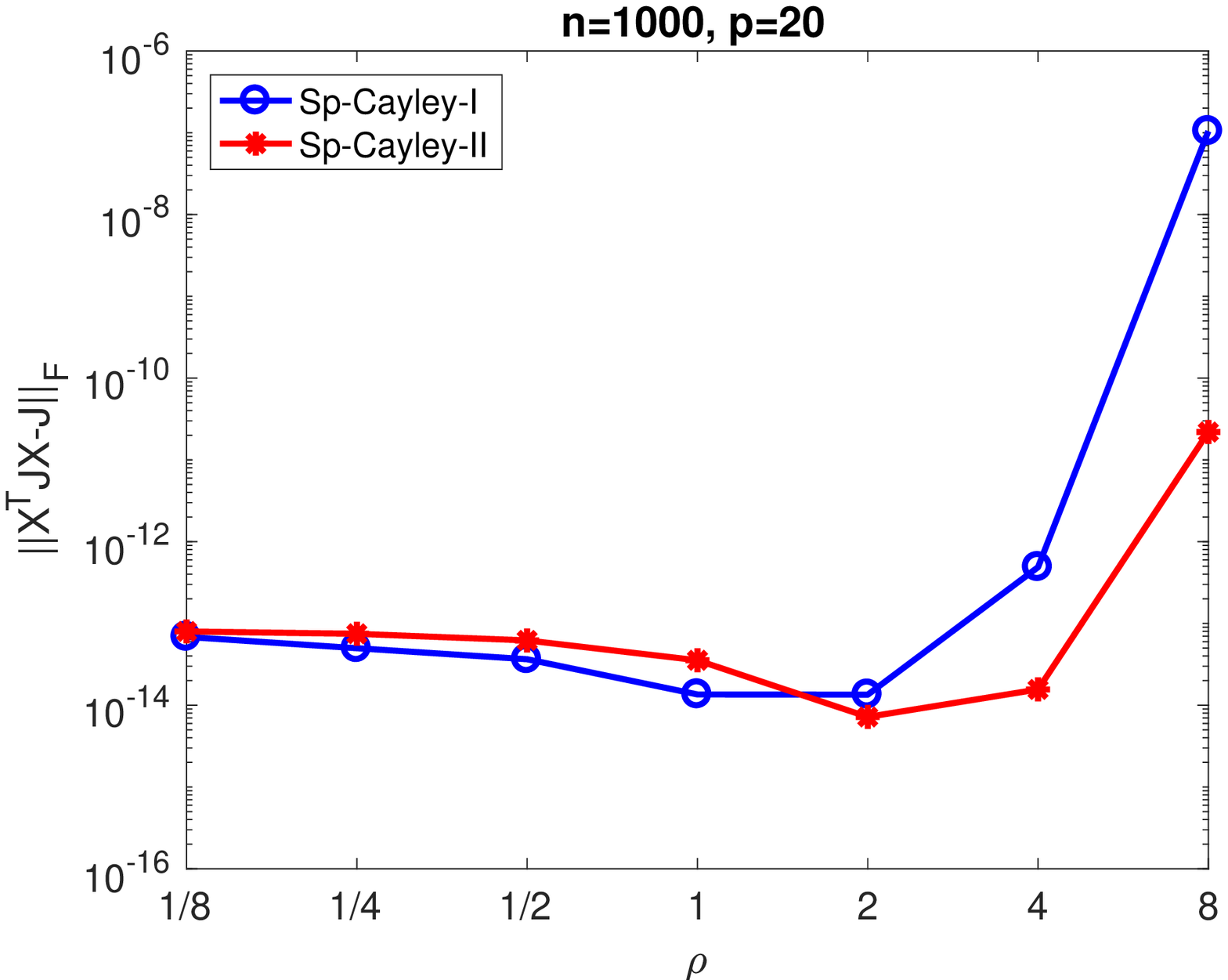}}
	\caption{A comparison of Sp-Cayley-I and Sp-Cayley-II with different parameter $\rho$\label{fig:metric}}
\end{figure}

\subsubsection{A comparison of different line search schemes}\label{subsec:line-search}
The non-monotone line search strategy (\cref{subsection:non-monotone}) depends on several parameters. The purpose of this section is to investigate which among those parameters have a significant impact on the performance of \crefalg{alg:non-monotone gradient}. 

First, we consider the choice of {the} trial step size $\gamma_k$ in \crefalg{alg:non-monotone gradient}. In our case, the ambient space is Euclidean. Thus, we can use the BB method proposed in \cite{WenYin2013} {and define}
\begin{align*}\label{eq:BB}
{\gamma_k^{\mathrm{BB}1}} := \frac{\jkh{S^{k-1},S^{k-1}}}{\abs{\jkh{S^{k-1},{Y^{k-1}}}}}, \quad
{\gamma_k^{\mathrm{BB}2}}  :=\frac{\abs{\jkh{S^{k-1},Y^{k-1}}}}{{\jkh{Y^{k-1},{Y^{k-1}}}}},
\end{align*}
where $S^{k-1} = X^k - X^{k-1}$ and $Y^{k-1} =\rgrad{X^k} -  \rgrad{X^{k-1}}$. Note that this differs from the Riemannian BB method in \cite{iannazzo2018riemannian} since it adopts the Euclidean inner product $\jkh{\cdot,\cdot}$ rather than $g_\rho$. 
This choice is cheaper in flops and we have observed that it speeds up the algorithm.
Owing to its efficiency, we further adopt the \emph{alternating} 
BB strategy \cite{Dai_Fletcher_2005} {and choose the trial step size as}
\begin{align}\label{eq:ABB}
\gamma_k^{\mathrm{ABB}}:= \left\{
\begin{array}{cl}
\gamma_k^{\mathrm{BB}1},& \mbox{for odd}~~k,\\
\gamma_k^{\mathrm{BB}2},& \mbox{for even}~~k.
\end{array}
\right.
\end{align}
We next 
compare $\gamma_k^{\mathrm{BB1}}$, $\gamma_k^{\mathrm{BB2}}$, $\gamma_k^{\mathrm{ABB}}$, and the step size 
\begin{equation*}
	\gamma_k^{\mathrm{M}}:=2\abs{\frac{f(X^k)-f(X^{k-1})}{\mathrm{D}f(X)[Z]}}
\end{equation*}
proposed in \cite[(3.60)]{nocedal2006numerical}, where $\gamma_k^{\mathrm{M}}$ is also used in the line search function in Manopt\footnote{A MATLAB toolbox for optimization on manifolds (available from \href{https://www.manopt.org/}{https://www.manopt.org/}).}~\cite{manopt}.
In this test, we opt for the monotone line search ($\alpha=0$) adopted with tolerances $\epsilon = 10^{-10}$, $\epsilon_x = 10^{-10}$, $\epsilon_f = 10^{-14}$.
\cref{fig:gamma} reveals
that the BB strategies greatly improve the performance of the Riemannian gradient method, and outperform the classical initial step size $\gamma^{\mathrm{M}}$
 in iteration number and function value decreasing. 
We have obtained similar results, omitted here, for Sp-Cayley-II. {Since} $\gamma_k^{\mathrm{ABB}}$ is the most efficient choice in this experiment, we employ it as our default setting henceforth.

\begin{figure}[htpb]
	\centering
	\subfigure[F-norm of Riemannian gradient]
	{\includegraphics[scale=.34]
		{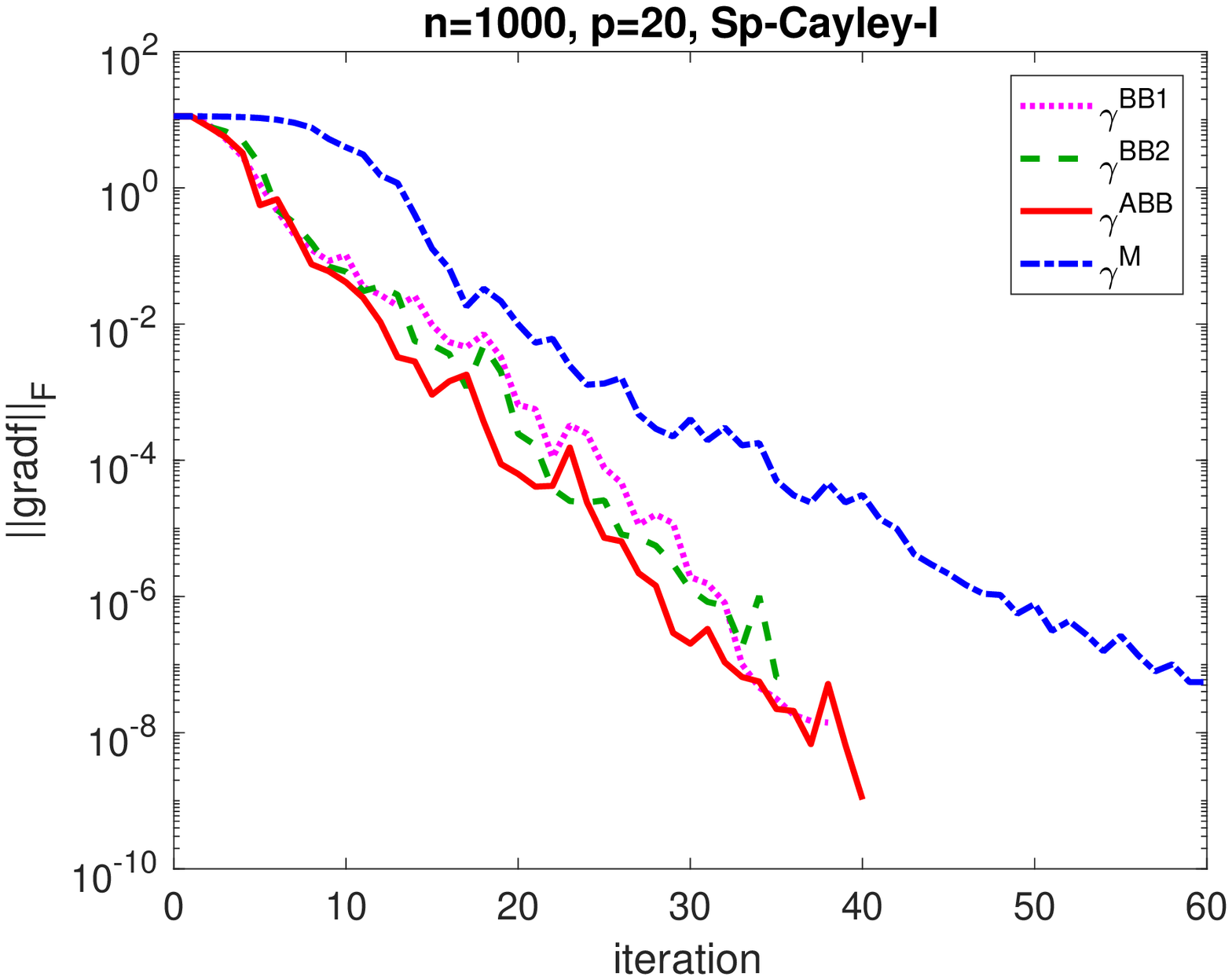}}
	\quad
	\subfigure[Function value]
	{\includegraphics[scale=.34]
		{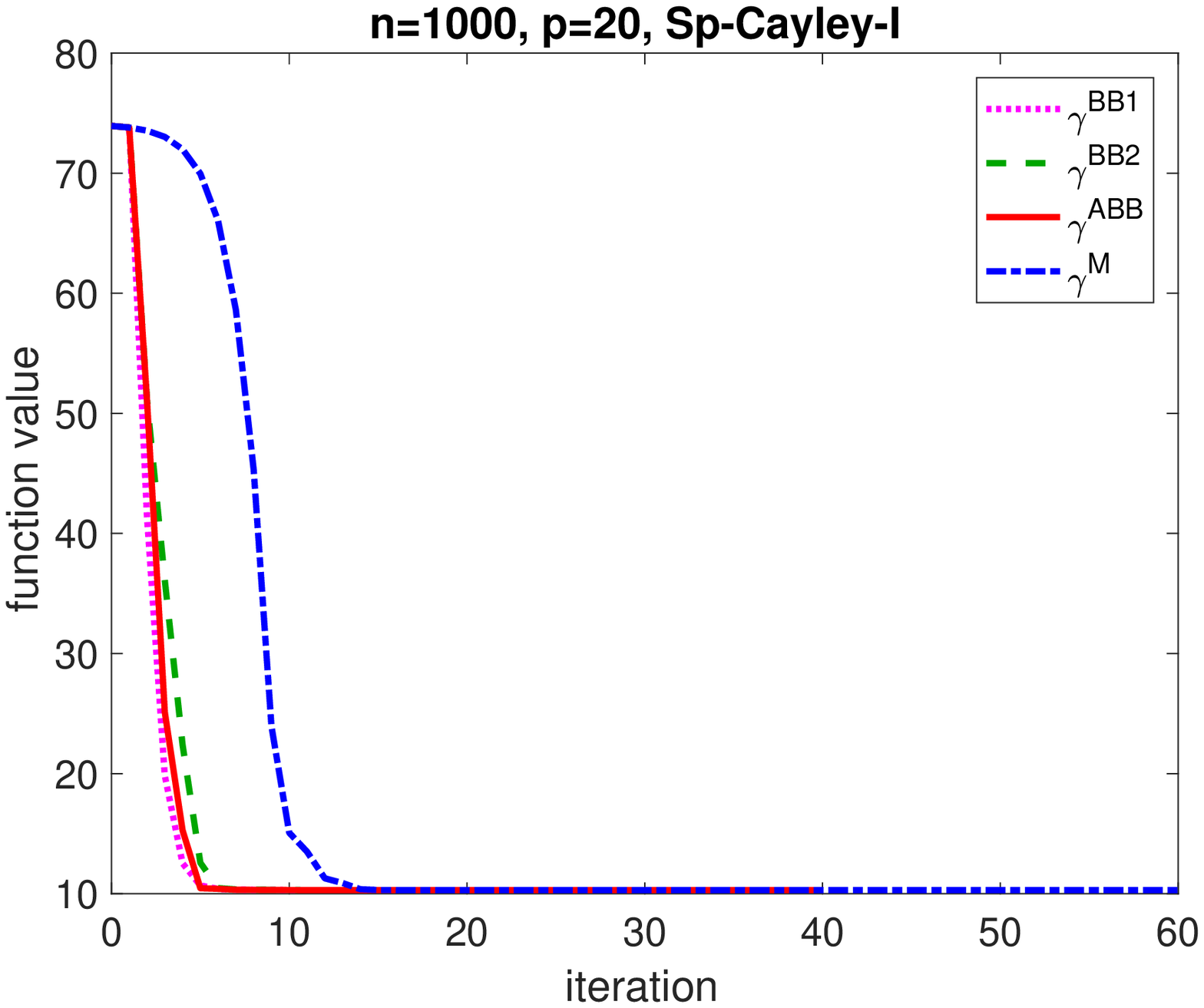}}
	\caption{A comparison {of} Sp-Cayley-I with different initial step size $\gamma_k$ {in the} monotone line search {scheme}\label{fig:gamma}}
\end{figure}

Next we investigate the impact of the parameter $\alpha$, which controls the degree of non-monotonicity. If $\alpha=0$, then condition~\eqref{eq:non-monotone} reduces to the usual monotone condition~\eqref{eq:monotone}. 
Here we scale the problem as \texttt{A=2*A/norm(A)} because we found that it reveals better the advantage that the non-monotone approach can have. We test Sp-Cayley-II
with $\alpha=0,0.85$, i.e., the monotone and non-monotone schemes.
The 
results are shown in \cref{fig:monotone}. 
The purpose of the non-monotone strategy is to make the line search condition~\eqref{eq:non-monotone} more prone than the monotone strategy to accept the trial step size $\gamma_k$, here $\gamma_k^{\mathrm{ABB}}$. We see that this results in a faster convergence in this experiment. Since the non-monotone condition \eqref{eq:non-monotone} works well in our problem, thus we choose it as a default setting henceforth. 

\begin{figure}[htpb]
	\centering
	\subfigure[F-norm of Riemannian gradient]
	{\includegraphics[scale=.34]
		{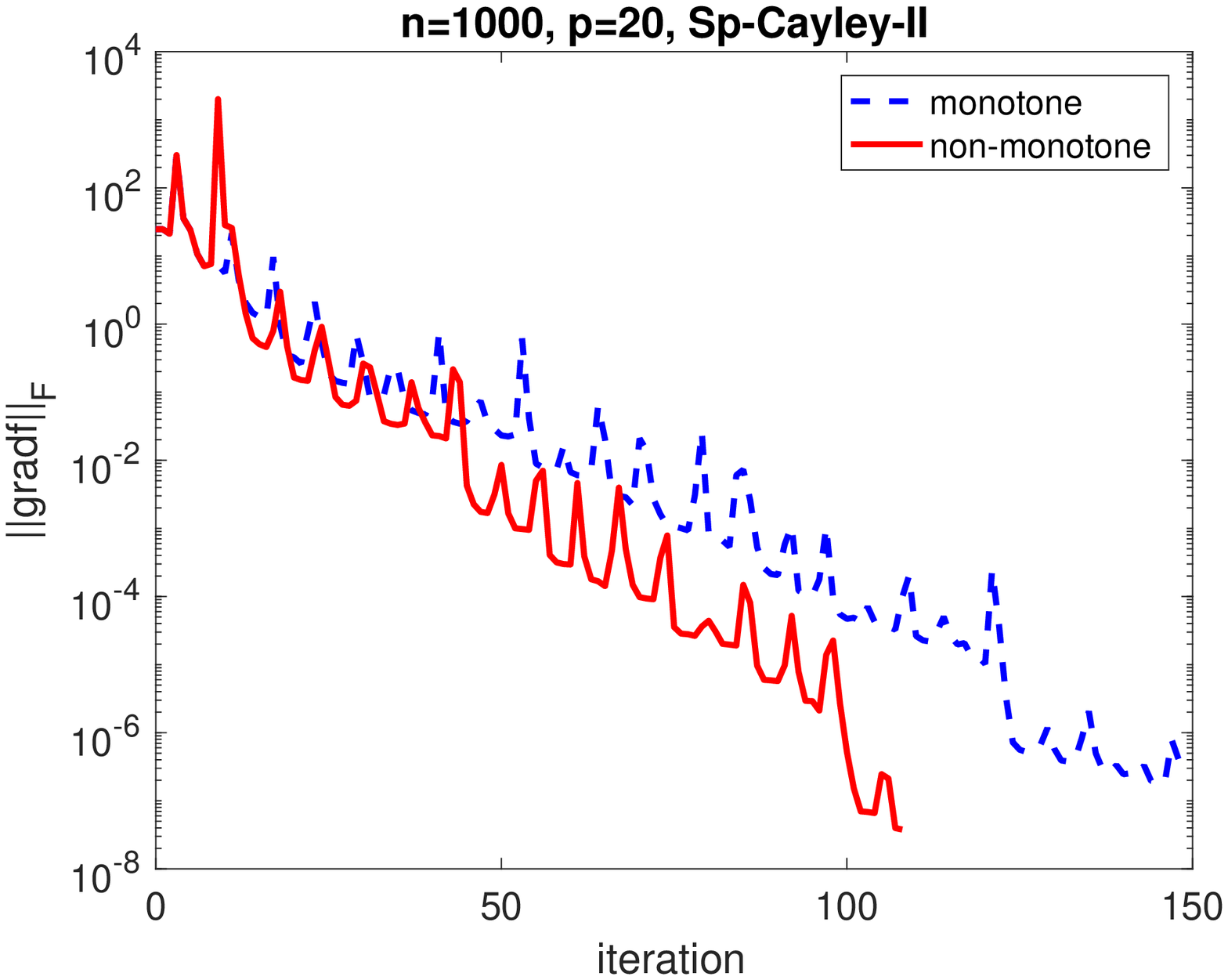}}
	\quad
	\subfigure[Function value]
	{\includegraphics[scale=.34]
		{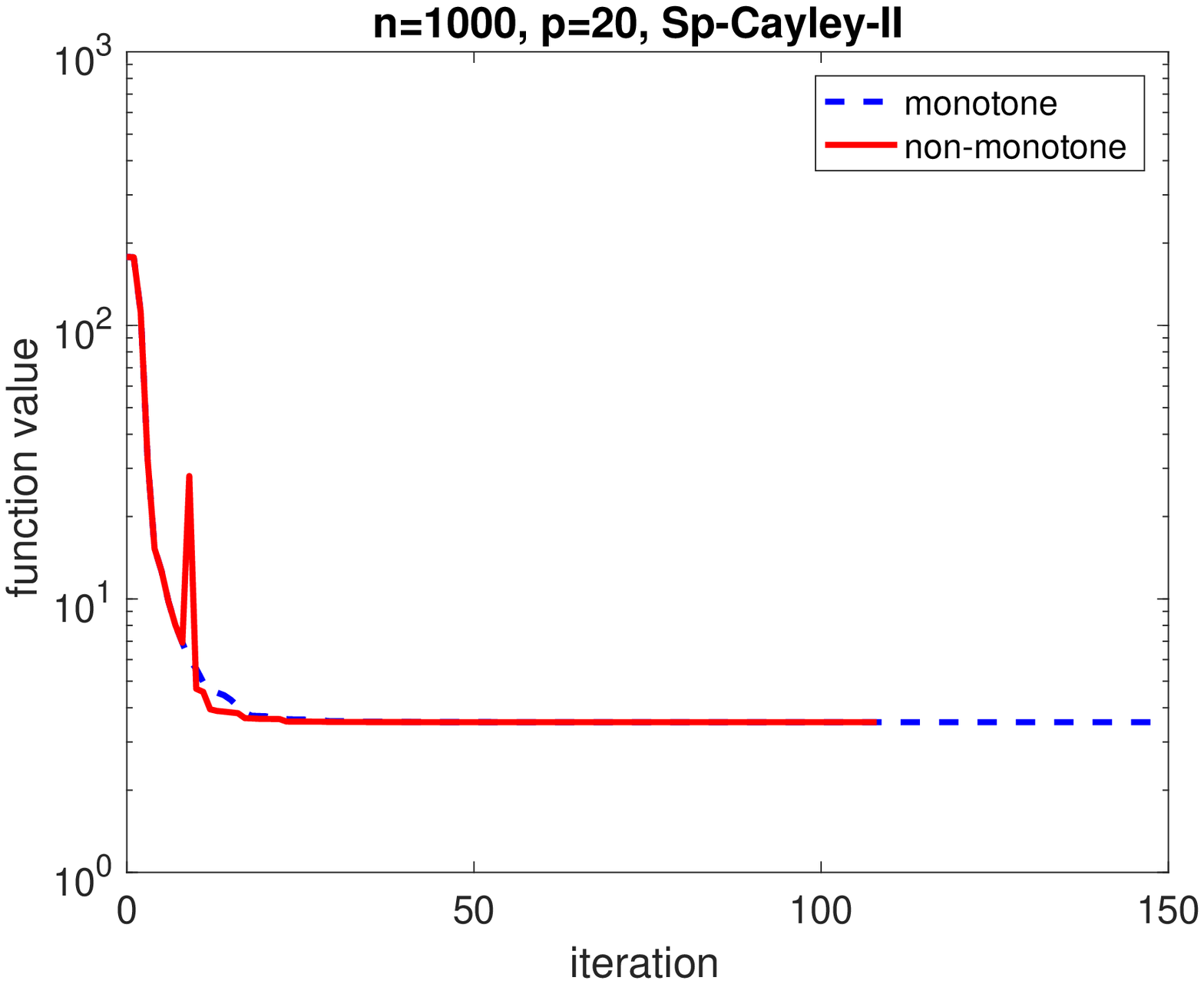}}
	\caption{A comparison {of} Sp-Cayley-II with {the} monotone and non-monotone line search {schemes}\label{fig:monotone}}
\end{figure}

\subsection{Nearest symplectic matrix problem}
In this section, we still focus on the nearest symplectic matrix problem \eqref{eq:nearest}. We first compare the algorithms Sp-Cayley-I and Sp-Cayley-II on an open matrix dataset: SuiteSparse Matrix Collection\footnote{Available from \href{https://sparse.tamu.edu/}{https://sparse.tamu.edu/}.}. Due to the different size and scale of data matrices, we choose the column number $p$ from the set $\{5,10,20,40,80\}$, and the target matrix $A\in\Rnnpp$ is generated by the first $2p$ columns of an original data matrix. In order to obtain a comparable error, we {normalize all matrices as $A/\|A\|_{\max}$}, where ${\|A\|_{\max}}:=\max_{i,j}\abs{A_{ij}}$. Numerical results are presented in \cref{tab:nearest sympectic matrix} for representative problem instances.
Here, ``fval" represents the function value, ``gradf", ``feasi", ``iter", and ``time" stand for $\norm{\mathrm{grad}_{{\rho}}f}\ff$, $\norm{X\zz JX-J}\ff$, the number of iteration{s}, and the wall-clock time in second{s}, respectively. 
From the table, we find that {both algorithms} perform well on most of the instances, and Sp-Cayley-I performs better than Sp-Cayley-II with fewer iteration number and fewer running time. In the largest problem ``2cubes\_sphere", {both} methods converge and obtain comparable results for the function value and gradient error. 
In addition, Sp-Cayley-II diverges on the problem ``msc23052" with $p=10$, while Sp-Cayley-I converges. Therefore, we conclude that for {the nearest symplectic matrix problem,} Sp-Cayley-I is more robust and efficient than {Sp-Cayley-II}.

\begin{table}[htbp]
	\scriptsize
	\centering
	\caption{{Numerical} results {for} the nearest symplectic matrix problem\label{tab:nearest sympectic matrix}} 
	\begin{tabular}{rrrrrrrrrrr}
		\hline\toprule
		\multirow{2}{*}{$p$} & Sp-Cayley-I &&&&&   Sp-Cayley-II &&&& \\\cmidrule(r){2-6}\cmidrule(r){7-11}
		& {fval}  &   {gradf}  & feasi & {iter}  & {time} & {fval}  &   {gradf}  & feasi & {iter}  & {time}   \\\midrule

		\multicolumn{11}{l}{2cubes\_sphere, $2n=101492$} \\ 
		5 & 	 3.995e+00 & 	 8.38e-05 & 	 1.39e-14 & 	 	  24 & 	 2.11 & 	 3.995e+00 & 	 1.40e-04 & 	 4.53e-15 & 	 	 23 & 	 2.18  \\ 
		10 & 	 7.331e+00 & 	 2.23e-04 & 	 1.95e-14 & 	  	 27 & 	 4.14 & 	 7.331e+00 & 	 5.28e-04 & 	 9.13e-15 & 	  	 28 & 	 4.60  \\ 
		20 & 	 1.602e+01 & 	 2.19e-04 & 	 4.76e-14 & 	  	 30 & 	 8.81 & 	 1.602e+01 & 	 6.72e-04 & 	 5.26e-14 & 	  	 36 & 	 12.53  \\ 
		40 & 	 3.423e+01 & 	 2.51e-03 & 	 7.08e-14 & 	  	 32 & 	 23.60 & 	 3.423e+01 & 	 7.20e-04 & 	 4.58e-14 & 	  	 40 & 	 31.79  \\ 
		80 & 	 1.056e+02 & 	 7.98e-04 & 	 4.05e-13 & 	  	 39 & 	 55.97 & 	 1.056e+02 & 	 1.50e-03 & 	 2.81e-10 & 	  	 41 & 	 71.11  \\

		
		
		\multicolumn{11}{l}{cvxbqp1, $2n=50000$} \\ 
		5 & 	 2.649e+00 & 	 2.03e-04 & 	 1.00e-14 & 	  	 22 & 	 0.78 & 	 2.649e+00 & 	 1.49e-03 & 	 3.15e-15 & 	  	 27 & 	 1.01  \\ 
		10 & 	 6.230e+00 & 	 2.16e-04 & 	 1.65e-14 & 	  	 26 & 	 1.54 & 	 6.230e+00 & 	 6.35e-04 & 	 6.35e-15 & 	 	 36 & 	 2.30  \\ 
		20 & 	 1.289e+01 & 	 2.86e-04 & 	 3.00e-14 & 	  	 26 & 	 2.79 & 	 1.289e+01 & 	 6.85e-04 & 	 1.08e-14 & 	  	 38 & 	 4.31  \\ 
		40 & 	 2.542e+01 & 	 1.33e-03 & 	 5.74e-14 & 	  	 28 & 	 8.19 & 	 2.542e+01 & 	 6.75e-04 & 	 3.72e-14 & 	  	 36 & 	 11.78  \\ 
		80 & 	 5.156e+01 & 	 6.38e-04 & 	 1.08e-13 & 	  	 26 & 	 19.81 & 	 5.156e+01 & 	 1.47e-03 & 	 7.33e-14 &     	 36 & 	 30.05  \\

		\multicolumn{11}{l}{msc23052, $2n=23052$} \\ 
		5 & 	 6.204e+00 & 	 1.23e-04 & 	 2.47e-04 & 	  	 92 & 	 1.30 & 	 6.204e+00 & 	 7.61e-04 & 	 5.26e-06 & 	  	 66 & 	 1.04  \\ 
		10 & 	 1.361e+01 & 	 5.87e-04 & 	 7.06e-08 & 	  	 67 & 	 1.66 & 	 4.556e-01 & 	 2.72e-02 & 	 4.24e+00 & 	  	 637 & 	 17.69  \\ 
		20 & 	 2.732e+01 & 	 9.82e-04 & 	 1.48e-04 & 	  	 80 & 	 3.70 & 	 2.730e+01 & 	 1.14e-02 & 	 8.69e-03 & 	  	 112 & 	 5.91  \\ 
		40 & 	 5.721e+01 & 	 2.18e-03 & 	 1.74e-06 & 	  	 70 & 	 6.78 & 	 5.721e+01 & 	 4.36e-03 & 	 1.98e-03 & 	  	 90 & 	 9.90  \\ 
		80 & 	 1.302e+02 & 	 9.96e-03 & 	 8.89e-03 & 	  	 98 & 	 27.87 & 	 1.302e+02 & 	 1.60e-02 & 	 1.00e-02 & 	  	 88 & 	 28.35  \\ 
		\multicolumn{11}{l}{Na5, $2n=5832$} \\ 
		5 & 	 2.170e+00 & 	 3.34e-04 & 	 1.29e-14 & 	  	 34 & 	 0.16 & 	 2.170e+00 & 	 1.05e-03 & 	 3.79e-15 & 	  	 35 & 	 0.17  \\ 
		10 & 	 4.378e+00 & 	 1.53e-04 & 	 3.62e-14 & 	  	 35 & 	 0.25 & 	 4.378e+00 & 	 4.34e-03 & 	 9.72e-15 & 	  	 41 & 	 0.29  \\ 
		20 & 	 9.033e+00 & 	 3.54e-04 & 	 6.24e-14 & 	  	 34 & 	 0.44 & 	 9.033e+00 & 	 1.46e-03 & 	 1.41e-14 & 	  	 37 & 	 0.53  \\ 
		40 & 	 1.828e+01 & 	 7.82e-04 & 	 8.58e-14 & 	  	 40 & 	 1.09 & 	 1.828e+01 & 	 1.75e-03 & 	 2.81e-14 & 	  	 44 & 	 1.24  \\ 
		80 & 	 3.731e+01 & 	 3.42e-04 & 	 3.32e-13 & 	  	 48 & 	 3.00 & 	 3.731e+01 & 	 6.74e-04 & 	 5.79e-14 & 	  	 56 & 	 3.92  \\ 
		\bottomrule\hline
	\end{tabular}
\end{table} 

We next compare our algorithms on randomly generated datasets. Given a set of samples $\{X_1,X_2,\dots,X_N\}$ with $X_i\in\symplectic$ for $i=1,\dots,N$, the extrinsic mean problem \cite[Section~3]{bhattacharya2003large} on $\symplectic$ is defined as
\begin{equation}\label{eq:minimal-distance}
\min\limits_{X\in\symplectic} \frac{1}{N}\sum_{i=1}^{N}\norm{X-X_i }\fs.
\end{equation} 
In view of~\cite[Section~3]{bhattacharya2003large}, the solutions of~\eqref{eq:minimal-distance} are those of~\eqref{eq:nearest} with $A=\frac{1}{N}\sum_{i=1}^{N}X_i$. This allows us to reuse the code that addressed~\eqref{eq:nearest}.
We test Sp-Cayley-I and Sp-Cayley-II for solving the problem \eqref{eq:minimal-distance} with three different random sample sets: (i) $N=100, n=p=2$; (ii) $N=100, n=p=10$; (iii) $N=1000, n=1000,p=20$. In each dataset, samples are randomly generated around a selected center $Y^0\in\symplectic$. Specifically, we choose $X_i=Y^0 e^{J(W_i+W_i\zz)}$, where \texttt{W$_i$=0.1*randn(2*p,2*p)}. The initial point $X^0$ and the center $Y^0$ are calculated by the strategy 3) (\cref{subsection:Stopping criteria}) for datasets (i)-(ii), and 2) for dataset (iii). In the first two sets (the symplectic group), Sp-Cayley-I and Sp-Cayley-II reduce to the same algorithm due to the same Riemannian gradient. Therefore, we omit the results of Sp-Cayley-II. We run our algorithms twice with different stopping tolerances: default and $\{\epsilon = 10^{-10}, \epsilon_x = 10^{-10}, \epsilon_f = 10^{-14}\}$. The detailed results for these two setting
are presented in \cref{tab:LS}. It reveals that our algorithms converge for three different sample sets with different stopping tolerances. Moreover, we show the initial and final errors of each sample in \cref{fig:LS}, where $X^*$ denotes the solution obtained by Sp-Cayley. From the figure, we observe that for both dataset (i) and (ii), the sample error greatly decreases with Sp-Cayley.

\begin{table}[htbp]
	\scriptsize
	\centering
	\caption{Numerical results for the extrinsic mean problem on $\symplectic$ \label{tab:LS}}
	\begin{tabular}{crrrrrrrrrr}
		\hline\toprule
		\multirow{2}{*}{$\epsilon$} & Sp-Cayley-I &&&&&   Sp-Cayley-II &&&& \\\cmidrule(r){2-6}\cmidrule(r){7-11}
		& {fval}  &   {gradf}  & feasi & {iter}  & {time} & {fval}  &   {gradf}  & feasi & {iter}  & {time}   \\\midrule
		
		\multicolumn{11}{l}{dataset (i), $n=2,p=2,N=100$} \\ 
		1e-05 & 1.627e+00  & 	 3.74e-05 & 	 1.07e-14 & 	  	 82 & 	 0.01  & - & - & - & - & 	 -  \\ 
		1e-10 & 1.627e+00 & 	 9.14e-10 & 	 8.84e-15 & 	  	 158 & 	 0.01  & - & - & - & - & 	 -  \\ 
		\multicolumn{11}{l}{dataset (ii), $n=10,p=10,N=100$} \\ 
		1e-05 & 3.068e+01 & 	 1.13e-04 & 	 8.97e-14 & 	  	 158 & 	 0.02  & - & - & - & - & 	 -  \\ 
		1e-10 & 3.068e+01 & 	 6.91e-09 & 	 1.12e-13 & 	  	 316 & 	 0.03    & - & - & - & - & 	 -  \\
		
		\multicolumn{11}{l}{dataset (iii), $n=1000,p=20,N=1000$} \\ 
		1e-05 & 	 1.333e+02 & 	 1.54e-04 & 	 3.56e-13 & 	  	 154 & 	 0.96 & 	 1.333e+02 & 	 2.54e-04 & 	 3.38e-13 & 	  	 178 & 	 1.05   \\ 
		1e-10 & 	 1.333e+02 & 	 3.67e-08 & 	 5.32e-13 & 	  	 256 & 	 1.33 & 	 1.333e+02 & 	 2.83e-07 & 	 4.63e-13 & 	  	 262 & 	 1.50  \\ 
		
		\bottomrule\hline
	\end{tabular}
\end{table} 

\begin{figure}[htbp]
	\centering
	\subfigure[Dataset (i), $\epsilon=10^{-10}$]
	{\includegraphics[scale=.35]
		{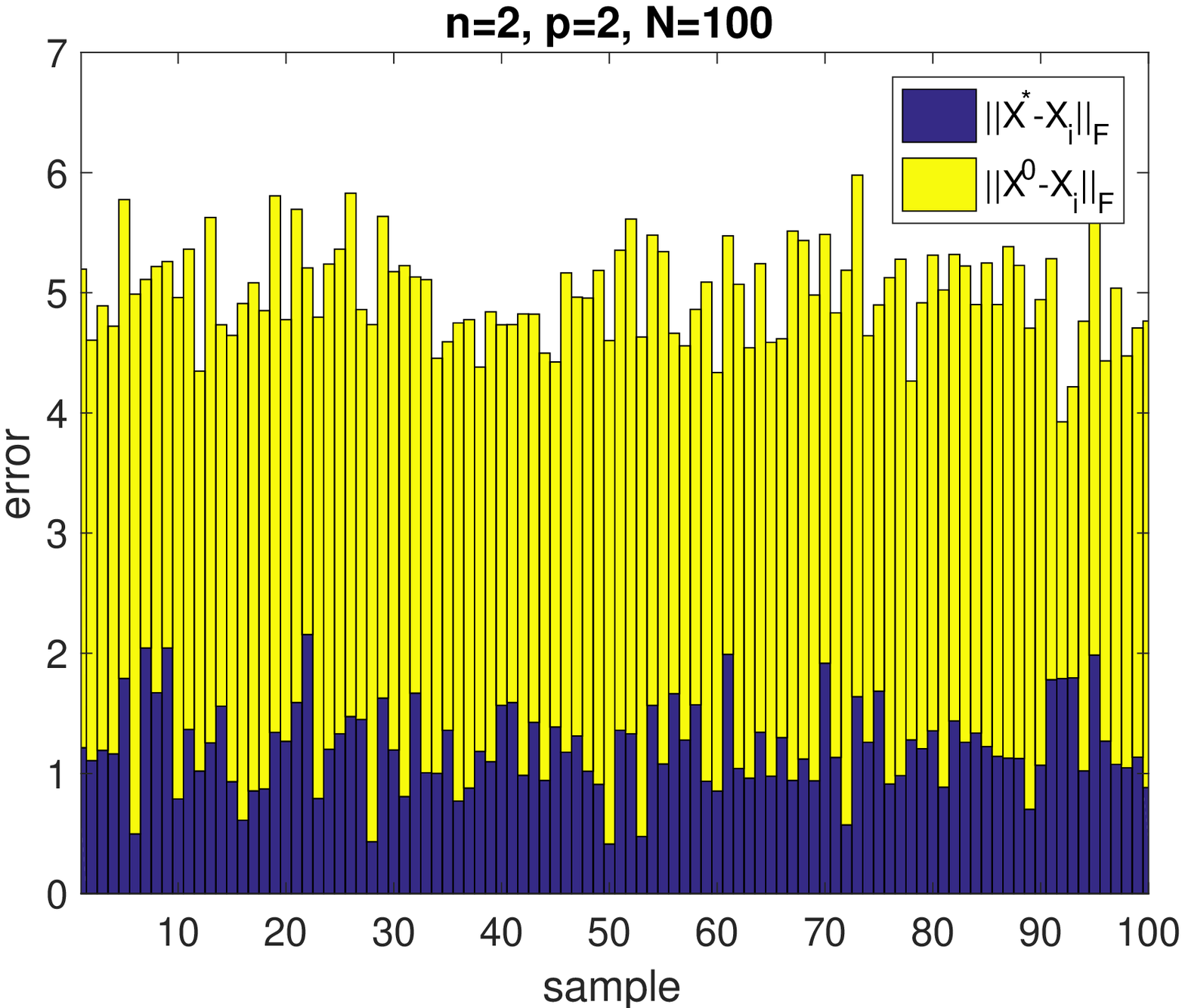}}
	\quad
	\subfigure[Dataset (ii), $\epsilon=10^{-10}$]
	{\includegraphics[scale=.35]
		{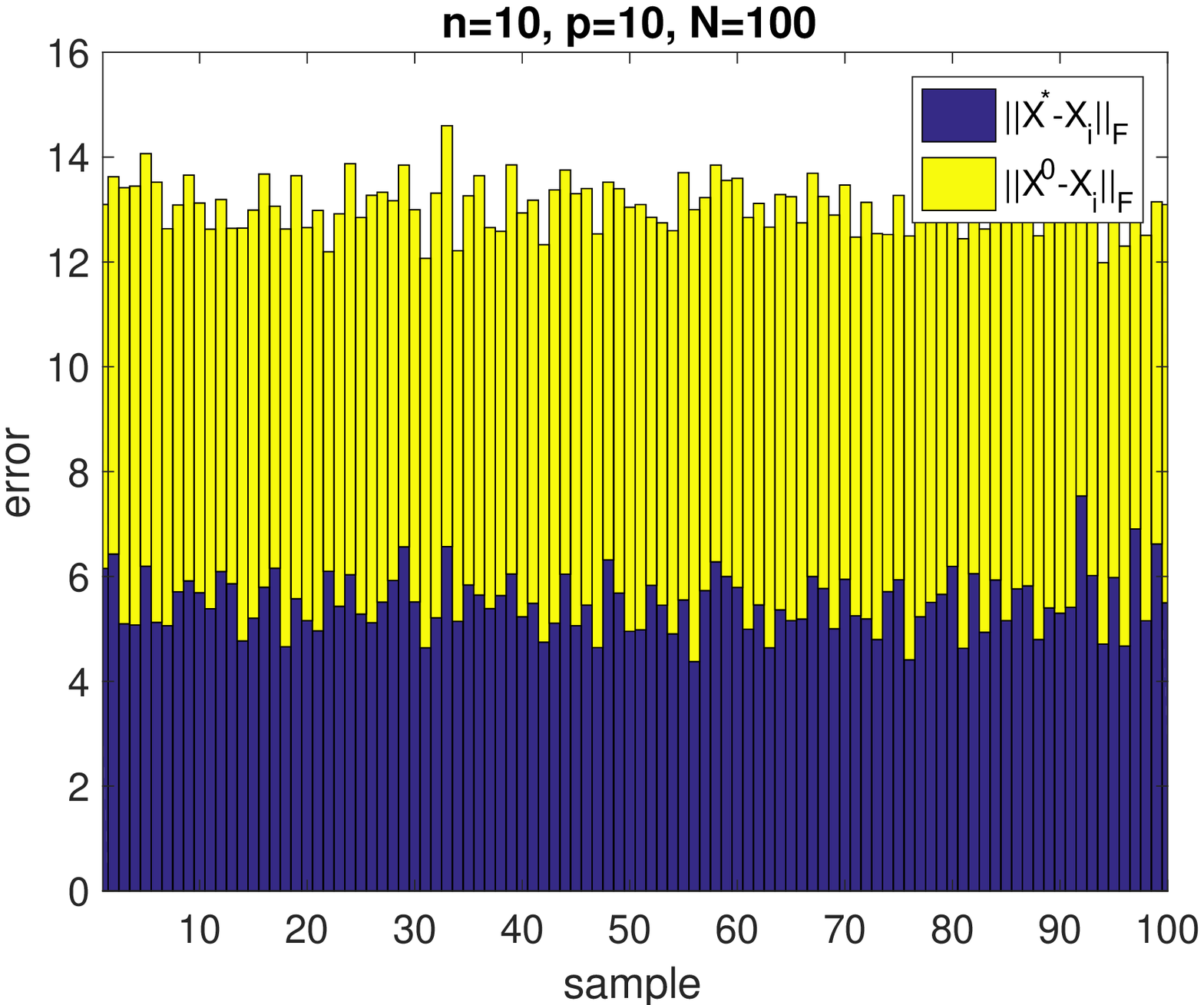}}
	\caption{A comparison of initial and final errors for each sample}
	\label{fig:LS}
\end{figure}


\subsection{Minimization of the Brockett cost function} \label{ssec:Brokett}
In \cite{brockett1989least}, Brockett investigated 
least squares matching problems on matrix Lie groups {with an~objective function}
\begin{equation}\label{eq:brockett}
f(X):=\tr(X\zz AXN-2BX\zz),
\end{equation}
where $A,N,B\in\Rnn$ are given matrices. This function is widely known as 
the Brockett cost function.  Recently, in \cite{machado2002optimization}, the results of \cite{brockett1989least} were extended to $P$-orthogonal matrices satisfying  $X\zz PX=P$ with a~given orthogonal matrix $P\in\Rnn$. For $P=J$, the problem reduces to an optimization problem on the symplectic group $\symplecticgroup$. Such a~problem was more recently considered in \cite{birtea2018optimization}, where several optimization algorithms were proposed. 

{In} this section, we study a~well-defined (bounded from below) minimization problem based on the Brockett cost function \eqref{eq:brockett} {with $N=I$ and $B=0$}. 
Specifically, 
we consider the following optimization problem
\begin{equation*}\label{eq:brockett-symplectic}
\min\limits_{X\in\symplectic} \tr(X\zz AX),
\end{equation*} 
where  $A\in\R^{2n\times 2n}$ is a symmetric positive definite matrix. Such a~problem arises, for example, in the symplectic eigenvalue problem \cite{williamson1936algebraic,bhatia2015symplectic} which will be investigated in more detail in \cref{ssec:sev}. In this test, the matrix $A\in \R^{2n\times 2n}$ is randomly generated {as \mbox{$A=Q\Lambda Q\zz$}}, where $Q\in \R^{2n\times 2n}$ is an orthogonal matrix computed from a QR factorization \texttt{Q=qr(randn(2*n,2*n)}), and $\Lambda\in \R^{2n\times 2n}$ is a diagonal matrix with diagonal elements $\Lambda_{ii}=\lambda^{1-i}$ for $i=1,2,\dots,2n$. 
The parameter $\lambda \ge 1$ determines the decay of eigenvalues of $A$. Th experiments are divided into two parts. At first, we solve the problem on the symplectic group, i.e., $n=p$, for matrices of different size $20\times20$, $80\times80$ and $160\times160$ and different parameters $\lambda\in\{1.01,1.04,1.07,1.1\}$. The numerical results are {presented} in \cref{tab:brockett}. From the table, we observe that Sp-Cayley (since Sp-Cayley-I and Sp-Cayley-II reduce to the same method when $n=p$) works well on different problems.
In the second part, we move to the symplectic Stiefel manifold and test {our algorithms on different problems with parameters} $n=1000,2000,3000$, $p=5,10,20,40,80$, and $\lambda$ {as above}. The corresponding results are also presented in \cref{tab:brockett}.
It illustrates that for relatively large problems, Sp-Cayley-I and Sp-Cayley-II still perform well, and have the comparable function values and feasibility violations. 

\begin{table}[htbp]
	\scriptsize
	\centering
	\caption{{Numerical} results in the Brockett cost function minimization\label{tab:brockett}} 
	\begin{tabular}{rrrrrrrrrrr}
		\hline\toprule
		\multirow{2}{*}{} & Sp-Cayley-I &&&&&   Sp-Cayley-II &&&& \\\cmidrule(r){2-6}\cmidrule(r){7-11}
		& {fval}  &   {gradf}  & feasi & {iter}  & {time} & {fval}  &   {gradf}  & feasi & {iter}  & {time}   \\\midrule
		
		${\lambda}$ &\multicolumn{10}{l}{$n=p=10$} \\
		1.01 & 	 3.642e+01 & 	 5.73e-06 & 	 3.98e-14 & 	  	 12 & 	 0.01 &		- &  - &  - &  - &  - \\
		1.04 & 	 2.793e+01 & 	 5.75e-06 & 	 5.19e-14 & 	  	 14 & 	 0.01 & 	 - &  - &  - &  - &  -\\
		1.07 & 	 2.168e+01 & 	 1.10e-05 & 	 2.20e-14 & 	  	 14 & 	 0.01 & 	   - &  - &  - &  - &  -\\
		1.10 & 	 1.737e+01 & 	 6.26e-06 & 	 4.85e-14 & 	  	 19 & 	 0.01 & 	   - &  - &  - &  - &  -\\
		 & \multicolumn{10}{l}{$n=p=40$} \\ 
		1.01 & 	 1.094e+02 & 	 4.25e-06 & 	 4.90e-13 & 	  	 14 & 	 0.01 & 	   - &  - &  - &  - &  -\\
		1.04 & 	 4.197e+01 & 	 1.61e-04 & 	 6.56e-13 & 	  	 25 & 	 0.02 & 	   - &  - &  - &  - &  -\\
		1.07 & 	 2.032e+01 & 	 4.69e-05 & 	 6.15e-13 & 	  	 46 & 	 0.04 & 	   - &  - &  - &  - &  -\\
		1.10 & 	 1.222e+01 & 	 1.00e-04 & 	 5.28e-13 & 	  	 72 & 	 0.07 & 	  - &  - &  - &  - &  -\\
		& \multicolumn{10}{l}{$n=p=80$} \\ 
		1.01 & 	 1.531e+02 & 	 2.98e-05 & 	 3.30e-12 & 	  	 19 & 	 0.07 & 	   - &  - &  - &  - &  -\\
		1.04 & 	 3.256e+01 & 	 1.46e-04 & 	 3.95e-12 & 	  	 59 & 	 0.21 & 	   - &  - &  - &  - &  -\\
		1.07 & 	 1.421e+01 & 	 2.06e-04 & 	 2.89e-12 & 	  	 154 & 	 0.58 & 	   - &  - &  - &  - &  -\\
		1.10 & 	 8.442e+00 & 	 1.90e-04 & 	 2.39e-12 & 	  	 252 & 	 0.94 & 	   - &  - &  - &  - &  -\\\midrule
		${p}$ & \multicolumn{10}{l}{$n=1000, \lambda=1.01$} \\ 
		5 & 	 3.150e-04 & 	 2.10e-04 & 	 1.20e-14 & 	  	 182 & 	 0.90 & 	 2.015e-04 & 	 1.62e-04 & 	 1.50e-14 & 	 	 208 & 	 1.05  \\ 
		10 & 	 3.631e-04 & 	 1.61e-04 & 	 3.07e-14 & 	  	 293 & 	 2.05 & 	 4.028e-04 & 	 1.95e-04 & 	 4.40e-14 & 	  	 260 & 	 1.93  \\ 
		20 & 	 5.902e-04 & 	 1.85e-04 & 	 9.02e-14 & 	 	 362 & 	 3.18 & 	 5.350e-04 & 	 1.81e-04 & 	 1.37e-13 & 	  	 280 & 	 2.97  \\ 
		40 & 	 7.764e-04 & 	 1.61e-04 & 	 3.82e-13 & 	  	 484 & 	 7.06 & 	 6.669e-04 & 	 1.52e-04 & 	 7.34e-13 & 	  	 548 & 	 9.77  \\ 
		80 & 	 1.037e-03 & 	 5.72e-04 & 	 1.65e-12 & 	  	 649 & 	 24.33 & 	 1.094e-03 & 	 1.72e-04 & 	 3.02e-12 & 	  	 619 & 	 21.14  \\ 
		& \multicolumn{10}{l}{$n=2000, \lambda=1.04$} \\ 
		10 & 	 1.241e-04 & 	 1.77e-04 & 	 2.95e-14 & 	  	 146 & 	 5.15 & 	 1.105e-04 & 	 1.85e-04 & 	 3.08e-14 & 	  	 174 & 	 6.22  \\ 
		20 & 	 1.748e-04 & 	 1.84e-04 & 	 2.12e-13 & 	  	 164 & 	 7.06 & 	 1.671e-04 & 	 2.06e-04 & 	 1.99e-13 & 	  	 234 & 	 10.61  \\ 
		40 & 	 2.216e-04 & 	 1.77e-04 & 	 3.37e-13 & 	  	 220 & 	 11.93 & 	 1.499e-04 & 	 1.36e-04 & 	 3.13e-13 & 	  	 270 & 	 17.54  \\ 
		80 & 	 2.948e-04 & 	 1.69e-04 & 	 1.91e-12 & 	  	 303 & 	 27.72 & 	 2.031e-04 & 	 1.47e-04 & 	 1.75e-12 & 	  	 273 & 	 24.48  \\ 
		& \multicolumn{10}{l}{$n=3000, \lambda=1.10$} \\ 
		5 & 	 1.794e-05 & 	 7.72e-04 & 	 9.63e-15 & 	  	 87 & 	 5.65 & 	 3.433e-05 & 	 1.70e-04 & 	 7.48e-15 & 	 	 74 & 	 4.98  \\ 
		10 & 	 5.133e-05 & 	 1.55e-04 & 	 5.01e-14 & 	  	 110 & 	 8.19 & 	 3.337e-05 & 	 2.34e-04 & 	 4.91e-14 & 	  	 125 & 	 8.65  \\ 
		20 & 	 7.352e-05 & 	 1.49e-04 & 	 1.58e-13 & 	  	 141 & 	 12.33 & 	 4.253e-05 & 	 1.32e-04 & 	 1.53e-13 & 	  	 142 & 	 12.41  \\ 
		40 & 	 1.226e-04 & 	 1.80e-04 & 	 4.43e-13 & 	  	 176 & 	 20.75 & 	 1.085e-04 & 	 1.91e-04 & 	 4.14e-13 & 	  	 134 & 	 16.43  \\ 
		80 & 	 1.661e-04 & 	 1.70e-04 & 	 3.96e-12 & 	  	 180 & 	 31.62 & 	 1.637e-04 & 	 1.91e-04 & 	 3.48e-12 & 	  	 188 & 	 35.02  \\ 
		\bottomrule\hline
	\end{tabular}
\end{table} 

\subsection{The {symplectic} eigenvalue {problem}} \label{ssec:sev}
It was shown in \cite{williamson1936algebraic} that for every
symmetric positive definite matrix $M\in\mathbb{R}^{2n\times 2n}$, there exists $X\in\symplecticgroup$ such that 
\begin{equation}
	X\zz MX = \begin{bmatrix} D & \\ 
	& D\end{bmatrix},\label{eq:sev}
\end{equation}
where $D =  {\mathrm{diag}}(d_1,\dots,d_n)$ and $0<d_1\le\dots\le d_n$. These entries are called {\em symplectic eigenvalues}, and \eqref{eq:sev} is referred to as the {\em symplectic eigenvalue problem}. Note that the symplectic eigenvalues are uniquely defined, whereas the symplectic transformation $X$ is not unique. It can be shown (see~\cite[(12)]{hiroshima2006}) that the symplectic eigenvalues of $M$ coincide with the positive (standard) eigenvalues of the matrix $G = \mathrm{i} J\zz M$, where $\mathrm{i}=\sqrt{-1}$ is the imaginary unit. 

In practice, it can be of interest to compute only a few extreme symplectic eigenvalues.
They can be determined by exploiting 
the following relationship between the $p\leq n$ smallest symplectic eigenvalues of $M$ and a symplectic optimization 
problem 
\begin{equation*}\label{eq:symplectic-eig}
2\sum_{j=1}^{p}d_j = \min_{X\in\symplectic} \tr(X\zz MX)
\end{equation*} 
which was first established in \cite{hiroshima2006} and further investigated in  \cite{bhatia2015symplectic}. Based on this relation, we aim to compute the smallest symplectic eigenvalue $d_1$ using \crefalg{alg:non-monotone gradient} with $p=1$. We test this algorithm on different data matrices from the MATLAB matrix gallery: (i) the Lehmer matrix; (ii) the Wilkinson matrix; (iii) the companion matrix of the polynomial whose coefficients are $1,\ldots,2n+1$; (iv) the central finite difference matrix. Whenever $M$ is not positive definite, which happens to the second and third case, we use $M\zz M$ instead of $M$ to generate the appropriate problem. The
parameters in \crefalg{alg:non-monotone gradient} are default settings. For a comparison, we {also} compute the smallest positive eigenvalue of $G$ by using the MATLAB function \texttt{eig}. The obtained results are shown in \cref{Tab:smallestSymplEig}. We observe that the symplectic eigenvalues computed by  Sp-Cayley-I are comparable with that provided by \texttt{eig}. 

\begin{table}[htbp]
	\small
	\centering
	\caption{The smallest symplectic eigenvalues\label{Tab:smallestSymplEig}}
	\begin{tabular}{lrrr}
		\hline\toprule
		Model matrix&  $n$ &   \texttt{eig} & Sp-Cayley-I \\ \midrule
		Lehmer & 50 &                  7.67480301454e-03 & 7.67480302204e-03  \\
		Wilkinson & 75 &             1.53471652403e+01 & 1.53471650305e+01  \\ 
		Companion & 500 & 		   5.47240371331e-02 & 5.47244189951e-02  \\
		Centr. Finite Diff. & 500 & 2.23005375485e-05&2.23005375834e-05  \\
		\bottomrule\hline
	\end{tabular}
\end{table} 

\section{Conclusion and perspectives}\label{sec:conclusion}
We have developed the ingredients---retraction and Riemannian gradient---that turn general first-order Riemannian optimization methods into concrete numerical algorithms for the optimization problem~\eqref{prob:original} on the symplectic Stiefel manifold $\symplectic$. The algorithms only need to be provided with functions that evaluate the objective function $f$ and the Euclidean gradient $\nabla \bar{f}$.
In order to cover the case of the Cayley retraction, we have extended the convergence analysis of a Riemannian non-monotone gradient descent method to encompass the situation where the retraction is not globally defined. This extended analysis leads to the conclusion that, for the sequences generated by the proposed algorithms, every accumulation point is a stationary point. Numerical experiments demonstrate the efficiency of the proposed algorithms. All the results in this paper apply to the symplectic group as the special case $p=n$.

This paper opens several perspectives for further research. In particular, it is tempting to further exploit the leeway in the choice of the metric and the retraction. Extensions to quotients of other quadratic Lie groups are also worth considering, as well as other applications.
\bibliographystyle{siamplain}
\bibliography{bibfile}

\begin{thebibliography}{10}

\bibitem{absil2009optimization}
{\sc P.-A. Absil, R.~Mahony, and R.~Sepulchre}, {\em Optimization Algorithms on
  Matrix Manifolds}, Princeton University Press, 2008,
  \url{https://press.princeton.edu/absil}.

\bibitem{ADM2002}
{\sc R.~L. Adler, J.-P. Dedieu, J.~Y. Margulies, M.~Martens, and M.~Shub}, {\em
  {N}ewton's method on {R}iemannian manifolds and a geometric model for the
  human spine}, IMA J. Numer. Anal., 22 (2002), pp.~359--390,
  \url{https://doi.org/10.1093/imanum/22.3.359}.

\bibitem{AfkhH17}
{\sc B.~Afkham and J.~Hesthaven}, {\em Structure preserving model of parametric
  {Hamiltonian} systems}, SIAM J. Sci. Comput., 39 (2017), pp.~A2616--A2644,
  \url{https://doi.org/10.1137/17M1111991}.

\bibitem{ajayi2013explicit}
{\sc D.~O.~A. Ajayi and A.~Banyaga}, {\em An explicit retraction of symplectic
  flag manifolds onto complex flag manifolds}, J. Geometry, 104 (2013),
  pp.~1--9, \url{https://doi.org/10.1007/s00022-013-0148-4}.

\bibitem{BB}
{\sc J.~Barzilai and J.~M. Borwein}, {\em Two-point step size gradient
  methods}, IMA J. Numer. Anal., 8 (1988), pp.~141--148,
  \url{https://doi.org/10.1093/imanum/8.1.141}.

\bibitem{Ber95}
{\sc D.~P. Bertsekas}, {\em Nonlinear Programming}, Athena Scientific, 1995,
  \url{http://www.athenasc.com/nonlinbook.html}.

\bibitem{bhatia2015symplectic}
{\sc R.~Bhatia and T.~Jain}, {\em On symplectic eigenvalues of positive
  definite matrices}, J. Math. Phys., 56 (2015), p.~112201,
  \url{https://doi.org/10.1063/1.4935852}.

\bibitem{bhattacharya_bhattacharya_2012}
{\sc A.~Bhattacharya and R.~Bhattacharya}, {\em Nonparametric Inference on
  Manifolds: With Applications to Shape Spaces}, Cambridge University Press,
  2012, \url{https://doi.org/10.1017/CBO9781139094764}.

\bibitem{bhattacharya2003large}
{\sc R.~Bhattacharya and V.~Patrangenaru}, {\em Large sample theory of
  intrinsic and extrinsic sample means on manifolds}, Ann. Statist., 31 (2003),
  pp.~1--29, \url{https://doi.org/10.1214/aos/1046294456}.

\bibitem{birtea2018optimization}
{\sc P.~Birtea, I.~Ca{\c{s}}u, and D.~Com{\u{a}}nescu}, {\em Optimization on
  the real symplectic group}, Monatsh. Math., 191 (2020), pp.~465--485,
  \url{https://doi.org/10.1007/s00605-020-01369-9}.

\bibitem{BouAbsCar2018}
{\sc N.~Boumal, P.-A. Absil, and C.~Cartis}, {\em Global rates of convergence
  for nonconvex optimization on manifolds}, IMA J. Numer. Anal., 39 (2018),
  pp.~1--33, \url{https://doi.org/10.1093/imanum/drx080}.

\bibitem{manopt}
{\sc N.~Boumal, B.~Mishra, P.-A. Absil, and R.~Sepulchre}, {\em {Manopt}, a
  {Matlab} toolbox for optimization on manifolds}, J. Machine Learn. Res., 15
  (2014), pp.~1455--1459, \url{https://www.manopt.org}.

\bibitem{brockett1989least}
{\sc R.~W. Brockett}, {\em Least squares matching problems}, Linear Algebra
  Appl., 122-124 (1989), pp.~761--777,
  \url{https://doi.org/10.1016/0024-3795(89)90675-7}.

\bibitem{BuchBH19}
{\sc P.~Buchfink, A.~Bhatt, and B.~Haasdonk}, {\em Symplectic model order
  reduction with non-orthonormal bases}, Math. Comput. Appl., 24 (2019),
  \url{https://doi.org/10.3390/mca24020043}.

\bibitem{Dai_Fletcher_2005}
{\sc Y.-H. Dai and R.~Fletcher}, {\em Projected {Barzilai}--{Borwein} methods
  for large-scale box-constrained quadratic programming}, Numer. Math., 100
  (2005), pp.~21--47, \url{https://doi.org/10.1007/s00211-004-0569-y}.

\bibitem{deGosson2006}
{\sc M.~A. de~Gosson}, {\em Symplectic Geometry and Quantum Mechanics},
  vol.~166 of Advances in Partial Differential Equations, Birkhäuser Basel,
  2006, \url{https://doi.org/10.1007/3-7643-7575-2}.

\bibitem{de2014metaplectic}
{\sc M.~A. de~Gosson and F.~Luef}, {\em Metaplectic group, symplectic {Cayley}
  transform, and fractional {Fourier} transforms}, J. Math. Anal. Appl., 416
  (2014), pp.~947--968, \url{https://doi.org/10.1016/j.jmaa.2014.03.013}.

\bibitem{draft1988lie}
{\sc A.~Draft, F.~Neri, G.~Rangarajan, D.~R. Douglas, L.~M. Healy, and R.~D.
  Ryne}, {\em {Lie} algebraic treatment of linear and nonlinear beam dynamics},
  Ann. Rev. Nuc. Part. Sci., 38 (1988), pp.~455--496,
  \url{https://doi.org/10.1146/annurev.ns.38.120188.002323}.

\bibitem{edelman1998geometry}
{\sc A.~Edelman, T.~A. Arias, and S.~T. Smith}, {\em The geometry of algorithms
  with orthogonality constraints}, SIAM J. Matrix Anal. Appl., 20 (1998),
  pp.~303--353, \url{https://doi.org/10.1137/S0895479895290954}.

\bibitem{fiori2011solving}
{\sc S.~Fiori}, {\em Solving minimal-distance problems over the manifold of
  real-symplectic matrices}, SIAM J. Matrix Anal. Appl., 32 (2011),
  pp.~938--968, \url{https://doi.org/10.1137/100817115}.

\bibitem{fiori2016riemannian}
{\sc S.~Fiori}, {\em A {Riemannian} steepest descent approach over the
  inhomogeneous symplectic group: Application to the averaging of linear
  optical systems}, Appl. Math. Comput., 283 (2016), pp.~251--264,
  \url{https://doi.org/10.1016/j.amc.2016.02.018}.

\bibitem{golub2013matrix}
{\sc G.~H. Golub and C.~F. Van~Loan}, {\em Matrix Computations}, Johns Hopkins
  University Press, 4th~ed., 2013,
  \url{https://jhupbooks.press.jhu.edu/title/matrix-computations}.

\bibitem{hairer2006geometric}
{\sc E.~Hairer, C.~Lubich, and G.~Wanner}, {\em Geometric Numerical
  Integration: Structure-Preserving Algorithms for Ordinary Differential
  Equations}, vol.~31, Springer Science \& Business Media, 2~ed., 2006,
  \url{https://doi.org/10.1007/3-540-30666-8}.

\bibitem{harris2004averageeye}
{\sc W.~Harris}, {\em The average eye}, Ophothal. Physiol. Opt., 24 (2004),
  pp.~580--585, \url{https://doi.org/10.1111/j.1475-1313.2004.00239.x}.

\bibitem{hiroshima2006}
{\sc T.~Hiroshima}, {\em Additivity and multiplicativity properties of some
  {Gaussian} channels for {Gaussian} inputs}, Phys. Rev. A, 73 (2006),
  p.~012330, \url{https://doi.org/10.1103/PhysRevA.73.012330}.

\bibitem{hsiang1968classification}
{\sc W.-y. Hsiang and J.~Su}, {\em On the classification of transitive
  effective actions on {Stiefel} manifolds}, Trans. Amer. Math. Soc., 130
  (1968), pp.~322--336,
  \url{https://doi.org/10.1090/S0002-9947-1968-0221529-1}.

\bibitem{hu2019brief}
{\sc J.~Hu, X.~Liu, Z.-W. Wen, and Y.-X. Yuan}, {\em A brief introduction to
  manifold optimization}, J. Oper. Res. Soc. China,  (2020),
  \url{https://doi.org/10.1007/s40305-020-00295-9}.

\bibitem{HuaAbsGal2017}
{\sc W.~Huang, P.-A. Absil, and K.~Gallivan}, {\em A {Riemannian} {BFGS} method
  without differentiated retraction for nonconvex optimization problems}, SIAM
  J. Optim., 28 (2018), pp.~470--495, \url{https://doi.org/10.1137/17M1127582}.

\bibitem{iannazzo2018riemannian}
{\sc B.~Iannazzo and M.~Porcelli}, {\em The {Riemannian} {Barzilai}--{Borwein}
  method with nonmonotone line search and the matrix geometric mean
  computation}, IMA J. Numer. Anal., 38 (2018), pp.~495--517,
  \url{https://doi.org/10.1093/imanum/drx015}.

\bibitem{machado2002optimization}
{\sc L.~M. Machado and F.~S. Leite}, {\em Optimization on quadratic matrix
  {Lie} groups},  (2002), \url{http://hdl.handle.net/10316/11446}.

\bibitem{meyer1992linear}
{\sc K.~R. Meyer and G.~R. Hall}, {\em Linear {Hamiltonian} Systems}, Springer
  New York, 1992, pp.~33--71,
  \url{https://doi.org/10.1007/978-1-4757-4073-8_2}.

\bibitem{nocedal2006numerical}
{\sc J.~Nocedal and S.~Wright}, {\em Numerical Optimization}, Springer Science
  \& Business Media, 2006, \url{https://doi.org/10.1007/978-0-387-40065-5}.

\bibitem{PengM16}
{\sc L.~Peng and K.~Mohseni}, {\em Symplectic model reduction of {Hamiltonian}
  systems}, SIAM J. Sci. Comput., 38 (2016), pp.~A1--A27,
  \url{https://doi.org/10.1137/140978922}.

\bibitem{RinWir2012}
{\sc W.~Ring and B.~Wirth}, {\em Optimization methods on {Riemannian} manifolds
  and their application to shape space}, SIAM J. Optim., 22 (2012),
  pp.~596--627, \url{https://doi.org/10.1137/11082885X}.

\bibitem{sigrist1973cross}
{\sc F.~Sigrist and U.~Suter}, {\em Cross-sections of symplectic {Stiefel}
  manifolds}, Trans. Amer. Math. Soc., 184 (1973), pp.~247--259,
  \url{https://doi.org/10.1090/S0002-9947-1973-0326728-8}.

\bibitem{wang2018riemannian}
{\sc J.~Wang, H.~Sun, and S.~Fiori}, {\em A {Riemannian}-steepest-descent
  approach for optimization on the real symplectic group}, Math. Meth. Appl.
  Sci., 41 (2018), pp.~4273--4286, \url{https://doi.org/10.1002/mma.4890}.

\bibitem{WenYin2013}
{\sc Z.~Wen and W.~Yin}, {\em A feasible method for optimization with
  orthogonality constraints}, Math. Program., 142 (2013), pp.~397--434,
  \url{https://doi.org/10.1007/s10107-012-0584-1}.

\bibitem{williamson1936algebraic}
{\sc J.~Williamson}, {\em On the algebraic problem concerning the normal forms
  of linear dynamical systems}, Amer. J. Math., 58 (1936), pp.~141--163,
  \url{https://doi.org/10.2307/2371062}.

\bibitem{wu2008optimal}
{\sc R.~Wu, R.~Chakrabarti, and H.~Rabitz}, {\em Optimal control theory for
  continuous-variable quantum gates}, Phys. Rev. A, 77 (2008), p.~052303,
  \url{https://doi.org/10.1103/PhysRevA.77.052303}.

\bibitem{wu2010critical}
{\sc R.-B. Wu, R.~Chakrabarti, and H.~Rabitz}, {\em Critical landscape topology
  for optimization on the symplectic group}, J. Optim. Theory Appl., 145
  (2010), pp.~387--406, \url{https://doi.org/10.1007/s10957-009-9641-1}.

\bibitem{zhang2004nonmonotone}
{\sc H.~Zhang and W.~W. Hager}, {\em A nonmonotone line search technique and
  its application to unconstrained optimization}, SIAM J. Optim., 14 (2004),
  pp.~1043--1056, \url{https://doi.org/10.1137/S1052623403428208}.

\end{thebibliography}
%
%
\end{document}